\documentclass[10pt,letterpaper,reqno]{amsart}
\usepackage{amsmath,amsfonts,amsthm,amssymb,amscd}
\usepackage{graphicx} 
\usepackage{upref} 
\usepackage{mathptmx}
\usepackage[T1]{fontenc}
\usepackage{bm}
\usepackage{calc}
\usepackage{caption}
\usepackage{hyperref}
\usepackage[usenames,dvipsnames]{xcolor}

\usepackage{mathrsfs}
\usepackage{changepage}
\usepackage{longtable}
\usepackage{enumerate,letltxmacro}
\LetLtxMacro\itemold\item

\newcommand{\inditem}{\itemindent05mm\itemold}

\newtheorem{theorem}{Theorem}[section]
\newtheorem*{theorem*}{Theorem}	

\newtheorem{corollary}[theorem]{Corollary}
\newtheorem{lemma}[theorem]{Lemma}
\newtheorem{proposition}[theorem]{Proposition}

\newtheorem{definition}[theorem]{Definition}

\newtheorem{remark}[theorem]{Remark}

\newcommand{\R}{\mathbb{R}}

\newcommand{\ti}{\tilde}
\newcommand{\ol}{\overline}
\newcommand{\tor}{ \left(\frac{t}{r}\right)}

\newcommand{ \eeq  }{\, = \,}
\newcommand{\neeq }{\, \neq \,}
\newcommand{\geeq}{\,  \ge  \,}
\newcommand{ \leeq}{\,  \le  \,}
\newcommand{\pl }{\, + \,}
\newcommand{\mn}{\, - \,}

\newcommand{\st}[1]{\, #1 \,}

\newcommand{\lrf}[2]{ \left( \frac{#1}{#2} \right)}

\makeatletter
\renewcommand*\env@matrix[1][c]{\hskip -\arraycolsep
   \leeq t\@ifnextchar\new@ifnextchar
  \array{*\c@MaxMatrixCols #1}}
\makeatother

\title{Isoperimetric Regions in $\R^n$ with Density $r^p$}
\author[W. Boyer]{Wyatt Boyer}
\author[B. Brown]{Bryan Brown}
\author[G. Chambers]{Gregory R. Chambers}
\author[A. Loving]{Alyssa Loving}
\author[S. Tammen]{Sarah Tammen}

\begin{document}

\begin{abstract}
We show that the unique isoperimetric regions in $\R^n$ with density $r^p$ for $n  \geeq  3$ and $p  \st{>} 0$ are balls with boundary through the origin.
\end{abstract}

\maketitle
\addcontentsline{toc}{section}{Acknowledgement}
\tableofcontents

\section{ Introduction}
Recently, there has been a surge of interest in manifolds with density, partly because of their role in Perelman's proof of the Poincar\'e Conjecture.  We consider the isoperimetric problem when volume and perimeter are weighted by the density function \( r^ p\) and prove the following theorem:\\ \\
\textbf{Theorem \ref{rp_thm}.} \emph{In  \( \R^n \) with density \( r^p \), where \( n  \geeq  3 \) and \( p  \st{>} 0 \), the unique isoperimetric regions, up to sets of measure zero, are balls with boundary through the origin.} 
\\

The density \( r^p \) is one of the simplest radial density functions, but it has some interesting properties.  First,  \( r^p \) is homogeneous in degree \( p \), which means that given an isoperimetric region of one volume, we can scale it to get an isoperimetric region of a different volume.  Second, \(r^p\) (or a constant multiple) is the only density for which spheres through the origin could be isoperimetric (see e.g. Rmk. \ref{r^p_only}).  We can view our present problem as a venture either to prove a partial converse of this statement in the case that \( p > 0 \) or to extend the work of Dahlberg et al., who proved the result in \( \R^2\) \cite[Thm. 3.16]{dahlberg}.
D\'iaz et al. \cite[Conj. 7.6]{diaz} conjectured the generalization to \( \R^n \) and reduced 
the problem to analyzing planar curves. Recently, 
Chambers \cite[Thm. 1.1]{chambers} proved that balls 
\emph{centered} at the origin are isoperimetric in \( \R^n \) 
with any radial \emph{log-convex} density.\\
\indent  We adapt Chambers' proof to density \( r^p \).  Like Chambers, we first consider an isoperimetric region that is spherically symmetric (see Defn. \ref{symmetrization_def}), then prove the result in the general case.  Given a spherically symmetric isoperimetric region, we prove that the generating curve for the boundary is a circle through the origin.  The behavior of this curve is determined by a differential equation corresponding to the fact that isoperimetric hypersurfaces have constant generalized mean curvature \cite[Defn. 2.3]{morgan_pratelli}.  By spherical symmetry and regularity, the rightmost point of the curve is on the \(e_1\)-axis, and the tangent vector at this point is vertical.  Our Lemmas \ref{cn_circ_gamma_circ} and \ref{kappa(0) equals lambda(0)} show that if the osculating circle at the rightmost point of the curve, which we may assume to be \( (1,0) \),
  goes through the origin, then the curve is a circle through the origin.\\
\indent   We suppose for contradiction that the initial osculating circle does not pass through the origin, then take two cases according to whether its center is right or left of \((1/2, 0) \).  We call these cases the right case and the left case, respectively.  In the right case, the curve is like that in Chambers' proof in that the curvature is greater at a point above the \(e_1\)-axis with tangent vector in the third quadrant than at the point of the same height with tangent vector in the second quadrant.  As a result, the curve has a vertical tangent before it meets the \(e_1\)-axis again and then curves in to meet the axis at an angle (Fig. \ref{fig:left_and_right}, right).  In the left case, the opposite inequality regarding curvatures holds, and, as a result, the curve never returns to vertical before reaching the axis (Fig. \ref{fig:left_and_right}, left).\\
\begin{figure}
	\begin{center}
		\includegraphics[scale =  0.49]{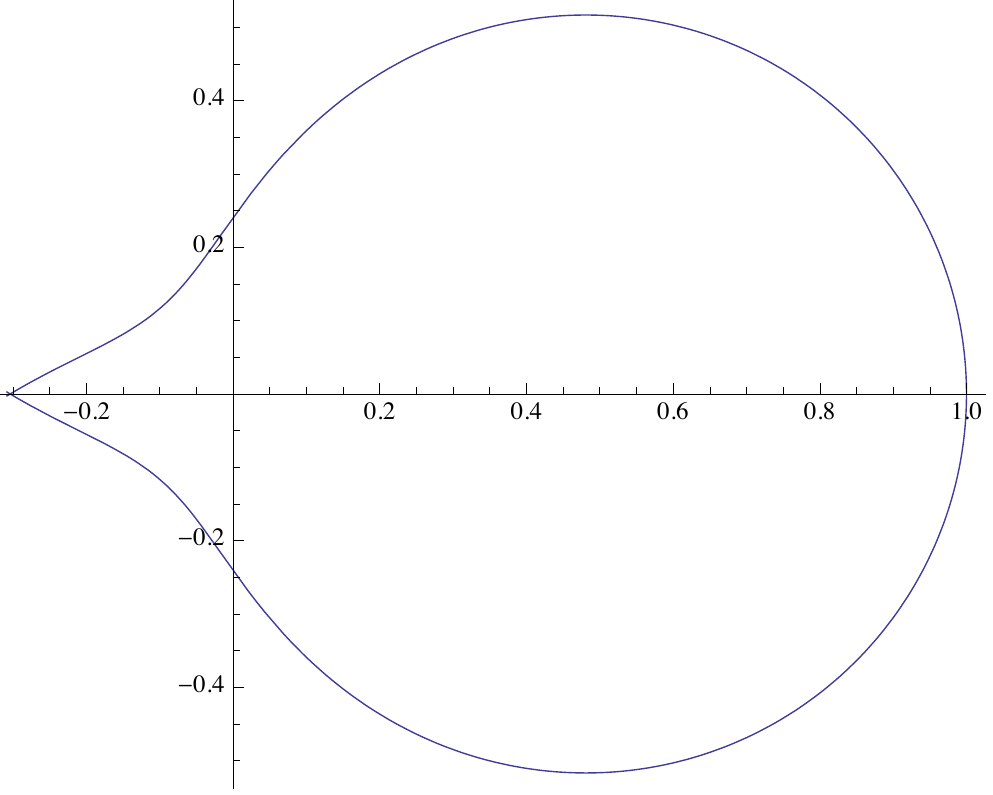}
		\hspace{2 cm}
	    	 \includegraphics[scale =  0.335]{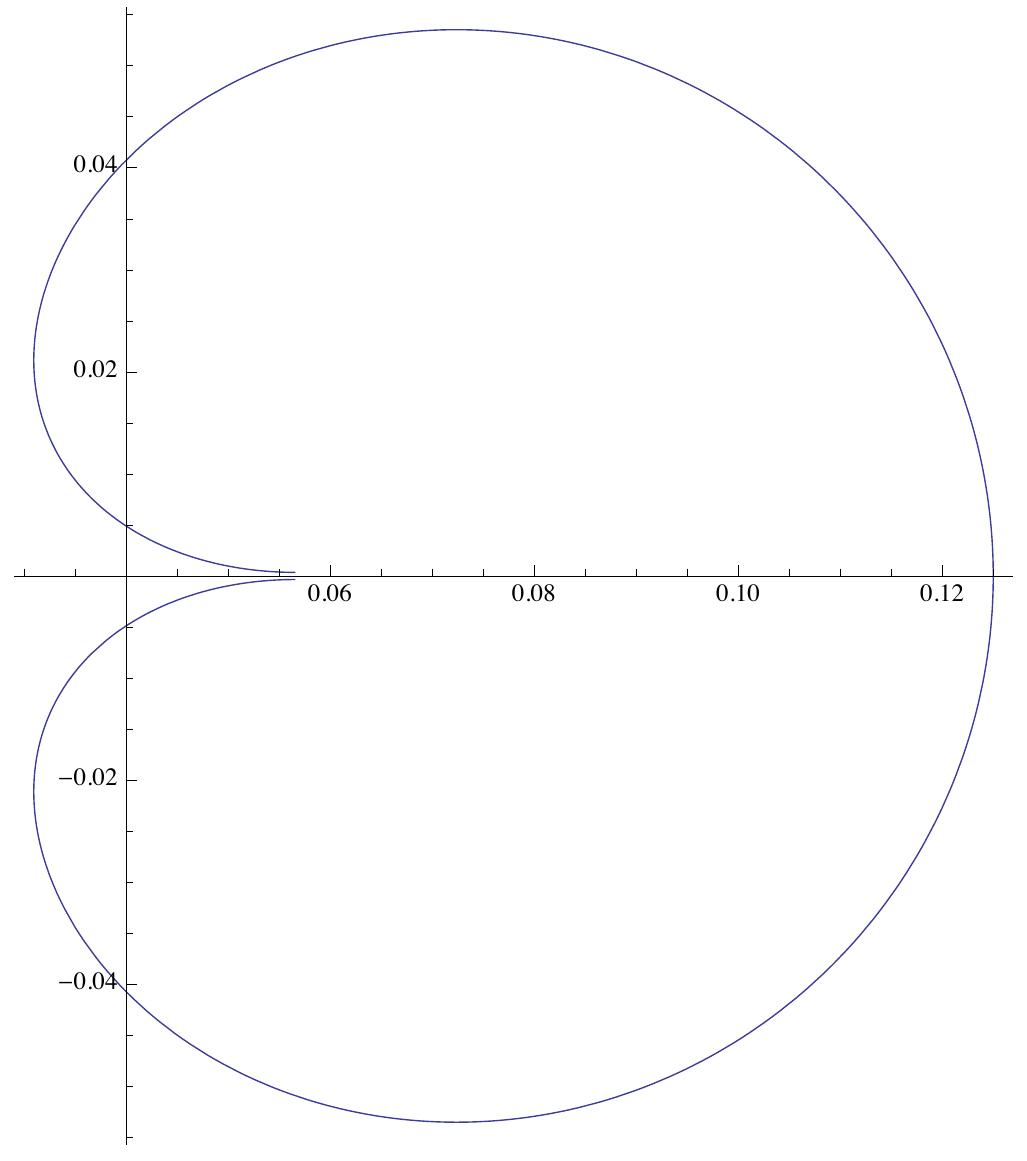}
		 \caption{Sample curves in the left and right cases \label{fig:left_and_right}}
	\end{center}
\end{figure}
\indent  The left case presents the new challenge of showing that there is only one point on the upper half of the curve where the tangent vector is horizontal (Prop. \ref{just_delta}).  Additionally, although the curve in the right case is similar to that in Chambers, the proof is different in that we do not have the hypothesis that an isoperimetric hypersurface is mean convex, which is what Chambers used to prove that curvature was positive on the final segment of the curve (\cite[Prop. 4.1]{chambers}).  We achieve the same result by computations that depend on the fact that our curve ends right of the \(e_2\)-axis (Lemma \ref{end_positive}), which is a property that may not hold for the generating curve in Chambers.\\

\section{Existence, Regularity, and Symmetry}

\begin{definition}
\emph{A} region \emph{$E$ is a measurable subset of $\R^n$.  Its} weighted volume
\emph{is the integral of the density over \(E\). Its} boundary \emph{is the 
topological boundary. Its} weighted perimeter \emph{is the integral of the density over the 
boundary with respect to \((n-1)\)-dimensional Hausdorff 
measure. We say a region is} isoperimetric \emph{if it minimizes weighted perimeter 
for fixed weighted volume.}
\end{definition}

Theorem \ref{existence}, a result of Morgan and Pratelli, guarantees the existence of isoperimetric regions of all volumes.  After defining a regular point (Defn. \ref{regular}), we state a standard result on the regularity of isoperimetric hypersurfaces.

\begin{theorem}
\label{existence}
 \cite[Thm. 3.3]{morgan_pratelli}
Assume that $f$ is a (lower-semicontinuous) radial density that diverges to infinity.  Then there exist isoperimetric sets of all volumes.
\end{theorem}

\begin{definition} 
\label{regular}
\emph{(Regular Point) Let $E$ be an isoperimetric region.  We say that a point $P \in \partial E$ is} regular \emph{if there is an open set $U$ containing $P$ so that $\partial E \cap U$ is a smooth, embedded $(n-1)$-dimensional manifold.}
\end{definition}

\begin{proposition}
\label{regularity result}
 \cite[Cor. 3.8]{morgan_regularity} Let $S$ be an $n$-dimensional isoperimetric hypersurface in a manifold $M$ with $C^{\,k-1, \, \alpha}$ ($k \geeq 1, \, 0 \st{<} \alpha \st{<}1$) and Lipschitz Riemannian metric.  Then except for a set of Hausdorff dimension at most $n-7$, $S$ is locally a $C^{\,k,\alpha}$ submanifold; real analytic if the metric is real analytic.
\end{proposition}

By \cite[Rmk. 3.10]{morgan_regularity}, the conclusion of Proposition \ref{regularity result} holds for a Riemannian manifold with density, provided that the density function is at least as smooth as the metric.  In our case, the density $r^p$ is smooth on $\R^{n} \mn \{0\}$.  Thus,
if $E \subset \R^n$ is an isoperimetric region for density $r^p$, then $\partial E$ is regular except on a set of Hausdorff dimension at most $n-8$, after perhaps altering $E$ by a negligible set of measure 0; henceforth we assume regions open.  By the first variation formula, generalized mean curvature is constant on the set of regular points.
The following proposition gives a sufficient condition for $\partial E$ to be regular at a point.

\begin{proposition}
\label{loc_halfspace_implies_regular}
If $P \in \partial E$ and $E$ locally lies in a half-space to one side of a hyperplane through $P$,  then $\partial E$ is regular at $P$, provided that the density function is positive at $P$.  
\end{proposition}

\begin{proof}
Since $E$ is an isoperimetric minimizer and the oriented tangent cone at $P$ lies in a halfspace, the 
oriented tangent cone is a hyperplane.  The result follows by \cite[Prop. 3.5, Rmk. 3.10]{morgan_regularity}.
\end{proof}

\begin{corollary}
\label{regular_max_magnitude}
All points in $\partial E$ of maximal distance from the origin are regular.  
\end{corollary}

\begin {definition}
\label{symmetrization_def}
\emph{(Spherical Symmetrization) Given a region $E \subset \R^n$, let $A_E(r)$ denote the area of the intersection of $E$ with $S_r$, the sphere of radius $r$ centered at the origin.  We define the \emph{spherical symmetrization} of $E$ to be the unique set $E^*$ such that for all $r  \st{>} 0$,  $A_E(r)  \eeq   A_{E^*}(r)$,
and $E^* \cap S_r$ is a closed spherical cap that passes through $(r,0, ...,0)$ and is rotationally symmetric about the $e_1$-axis.}
\end{definition}

\begin{remark}
\label{off-axis_regularity}
\emph{Since the set of singularities on the boundary of an isoperimetric region $E \subset \R^n$ has dimension at most $(n-8)$, it follows that if $E$ is spherically symmetric about the $e_1$-axis, then all points in $\partial E$ that are not on the $e_1$-axis are regular.}
\end{remark}

The following  theorem demonstrates that for a radial density, spherical symmetrization preserves weighted volume but does not increase weighted perimeter. Moreover there are certain conditions under which the perimeter of a region remains the same after symmetrization only if the original region was spherically symmetric about some (oriented) line through the origin.

\begin{theorem}
\cite[Thm. 6.2]{morgan_pratelli}
\label{symmetrization}
Let $f$ be a radial density on $\R^n$, and let E be a set of finite volume.  Then the spherical symmetrization $E^*$ satisfies 
$$|E^*|  \eeq   |E|$$ and $$P(E^*)  \leeq  P(E).$$
Suppose further that $E$ is an open set of finite perimeter, and let $\nu(x)$ denote the normal vector at any $x \in \partial E$.  If 
\( \mathscr{H}^{n-1}  \left(x \in \partial E: \nu(x)  \eeq   \pm \frac{x}{|x|} \right)  \eeq   0, \)
and the set 
\( I_E : \eeq  \{r \st{>}0: 0  \st{<}  \mathscr{H}^{n-1}(E \cap S_r)  \st{<}  \mathscr{H}^{n-1}(S_r)\} \)
is an interval, then $P(E^*)  \eeq   P(E)$ if and only if $E  \eeq   E^*$ up to rotation about the origin.
\end{theorem}

It is immediate that if $E$ is an isoperimetric region in Euclidean space with a radial density, then $E^*$ is also isoperimetric.

\section{Spheres Through The Origin Are Uniquely Minimizing}
\indent To prove our main result, Theorem \ref{rp_thm}, we begin by showing that any spherically symmetric isoperimetric region is a ball whose boundary is a sphere through the origin (Prop. \ref{symmetrization_sphere}).  The proof of Proposition \ref{symmetrization_sphere} comprises most of the paper, but we provide a sketch below. We apply this proposition to the symmetrized version of an arbitrary isoperimetric region to show that, in fact, any isoperimetric region is spherically symmetric about some oriented line through the origin (Prop. \ref{iso_implies_symmetric}).
\begin{proposition}
\label{symmetrization_sphere}
Suppose that \( E \subset \R^n \) is a spherically symmetric isoperimetric region in \( \R^n \) with density \( r^p \).  Then \( E \) is a ball whose boundary goes through the origin.
\end{proposition}

\begin{proof} 
Assume without loss of generality that  \( E \) is spherically symmetric about the positive \( e_1 \)-axis.  Then \( E \) can be generated by rotating a planar set \( A \) about the \(e_1\)-axis.  Since \( E \) is spherically symmetric about the positive \( e_1\)-axis,  \(\, A \) is also spherically symmetric about the positive
 \( e_1\)-axis.  By regularity of \( \partial E \) (Defn. \ref{regular}), we are assuming that \( A \) is open and that its boundary is a curve (possibly having multiple connected components). 
 We define \( \gamma \subset \partial A \) by beginning at the rightmost point on \( \partial A \) and following the curve through this point in both directions until it intersects the \( e_1 \)-axis again.   This definition relies on regularity properties of \( \partial E \); see the beginning of Section 4 for more details. 
 
 We assume that \(  \gamma = (\gamma_{\, 1}, \gamma_{\, 2}) : [ -\beta, \beta ] \st{\to} \R^2 \)  is an 
arclength parameterization so that \( \gamma(0) \) is the rightmost point on \( \partial A \) and \( \gamma(\pm \beta) \) is the other intersection of \( \gamma \) with the \( e_1 \)-axis. Since \( r^p \) is homogeneous, all isoperimetric  regions are similar, and we can assume without loss of generality that \( \gamma(0)  \eeq   (1,0) \).  
We will show that \( \gamma \) is a circle through the origin.  Given that \( \gamma \) is a circle through the origin, \( \gamma \) must comprise all of \( \partial A \) by spherical symmetrization.
 By Lemma \ref{cn_circ_gamma_circ}, to prove that \( \gamma \) is a circle through the origin, it suffices to prove that there exists an \( s \) so that the associated canonical circle \( C_s \) (see Defn. \ref{canonical_circle}) has the same curvature as \( \gamma \) at  \( \gamma(s) \) and \( C_s \) goes through the origin.  By Lemma  \ref{kappa(0) equals lambda(0)}, the canonical circle \( C_0 \) at the rightmost point has the same curvature as \( \gamma \) at \( \gamma(0) \).  Therefore, it suffices to prove that $C_0$ passes through the origin, which occurs if and only if the center of $C_0$ is $(1/2,0)$.\\
\indent Suppose that the center of $C_0$ is right of $(1/2,0)$.  By Proposition \ref{right_tangent}, $\gamma_{\, 1}(\beta)  > 0$ and $\lim_{s \to \beta^-} \gamma_{\, 1}\,'(s)  > 0$.  As a result, there exists $\varepsilon  \st{>} 0$ so that $\gamma \st{\cdot} \gamma\,'  \st{>} 0$ on $(\beta - \varepsilon, \beta)$, contradicting Lemma \ref{tangent_restriction}, which is a consequence of spherical symmetry. \\
\indent Now suppose that the center of $C_0$ is left of $(1/2,0)$.  By Proposition \ref{left_tangent}, $\gamma_{\, 1}(\beta)  <  0$ and $\lim_{s \to \beta^-} \gamma_{\, 1}\,'(s)  <  0$, which results in the same contradiction of spherical symmetry.\\
\indent The only remaining possibility is that  $\gamma$ is a circle through the origin.  Thus, \( \gamma  \eeq   \partial A \) and, when rotated, \( \gamma \) generates a sphere through the origin.
\end{proof}

Given Proposition \ref{symmetrization_sphere}, we can prove our claim that any isoperimetric region in \( \R^n \) with density $r^p$ is spherically symmetric.
 \begin{proposition}
\label{iso_implies_symmetric}
If $E$ is an isoperimetric region in $\R^n$ with density $r^p$, then $E  \eeq   E^*$, up to a rotation about the origin.
\end{proposition}

\begin{proof}
By regularity (Defn. \ref{regular}), we are assuming $E$ is open.
By Theorem \ref{symmetrization}, it suffices to show that $I_E$ is an interval and that
\begin{equation}
\label{area_zero}
\mathscr{H}^{n-1}  \left(x \in \partial E: \nu(x)  \eeq   \pm \frac{x}{|x|} \right)  \eeq   0.
\end{equation}
We call a point \( x \) with \( \nu(x) \eeq \pm \, x/|x| \) \emph{tangential}.
Since symmetrization (Defn. \ref{symmetrization_def}) preserves weighted volume without increasing weighted perimeter, $E^*$ is also isoperimetric.  Applying Proposition \ref{symmetrization_sphere}, we conclude that $E^*$ is a ball with
boundary through the origin. It follows that $I_E$ is an interval.  Moreover, there exists no $r \st{>} 0$ such that the spherical cap $S_r \cap E$ is a full sphere.  This will be important in our proof of (\ref{area_zero}). 
 Suppose for contradiction that there exists a positive area subset of $\partial E$ that is tangential. As in Morgan-Pratelli \cite[Pf. of Cor. 6.4]{morgan_pratelli}, at any smooth point of density of this tangential subset of $\partial E$, $\partial E$ has the same generalized mean curvature as a sphere 
 centered at the origin.  It follows by uniqueness of solutions to elliptic partial differential equations that a component of $\partial E$ is a sphere centered at the origin. 
 $E$ must contain an annular region centered at the origin with this spherical component as one of its bounding components.  Thus, there exists an interval $(r_0, r_1)$ such that for any $r$ in $(r_0, r_1)$, $S_r \cap E$ is a full sphere, contradicting the fact that the boundary of $E^*$ is a sphere through the origin.
\end{proof} 
Combining Propositions \ref{symmetrization_sphere} and \ref{iso_implies_symmetric} along with Theorem \ref{existence}, we have proved:

\begin{theorem}
\label{rp_thm}
In  \( \R^n \) with density \( r^p \), where \( n  \geeq  3 \) and \( p  \st{>} 0 \), the unique isoperimetric regions, up to sets of measure zero, are balls with boundary through the origin.
\end{theorem}

\section{Structure of Proof}
\indent 
 Sections 5, 6, and 7 are devoted to filling in the details of the proof of Proposition \ref{symmetrization_sphere}. 
Throughout these sections, we work within the following framework:

 Let $E$ be a spherically symmetric isoperimetric region.  Then there is a set $A \subset \R^2$ such that $E$ is the rotation of $A$ about the $e_1$-axis.  We will analyze a certain curve on the boundary of $A$.  We begin at the point $P$ on the $e_1$-axis that is the rightmost point on $\partial A$.
By spherical symmetry, $P$ is a point of $E$ farthest from the origin, so $\partial E$ is regular at $P$ by Corollary \ref{regular_max_magnitude}.  The tangent space to $\partial A$ at $P$ is spanned by $e_2$.  We follow $\partial A$, which has finite length, in both directions until it intersects the $e_1$-axis at another point.  The result is a Jordan curve 
 $\gamma(s): [-\beta, \beta] \to \R^2$ such that $\gamma(0)  \eeq   P$ and $\gamma(\pm \beta)$ is the other intersection of the curve with the $e_1$-axis (Fig. \ref{fig:gamma}).  Since $r^p$ is homogeneous, all isoperimetric regions are similar to each other.  Therefore, we may assume without loss of generality that $P  \eeq   (1,0)$.  We assume that $\gamma$ is a counterclockwise  arclength parameterization. Let $\gamma_{\, 1}$ and $\gamma_{\,2}$ denote the coordinates of $\gamma$.  Then $\gamma_{\, 1}(-s)  \eeq   \gamma_{\, 1}(s)$ and $\gamma_{\,2}(-s)  \eeq   -\gamma_{\,2}(s)$ for all $s$.  We let $\kappa(s)$ denote the curvature of $\gamma$ at $\gamma(s)$. \\
\indent  By Corollary \ref{regular_max_magnitude}, \(\gamma \) is smooth at \(0\).  By Remark \ref{off-axis_regularity}, \( \gamma \) is smooth at all remaining points in \( (-\beta, \beta) \).  Since \( \gamma \) is smooth at \(0 \) and \(0 \) is a global maximum point of \( \gamma_{\, 1}\), it follows that
\(\gamma\,'(0)  \eeq   (0, 1) \) and that $\kappa(0) \geeq 0$.  In fact, $0$ is a strict maximum point of $\gamma_{\,1}$.  To prove so, note that if there were an $s \neeq 0$ so that $\gamma_{\,1}(s) \st{>} \gamma_{\,1}(0)$, then it would also be the case that $|\gamma(s)| \st{>} |\gamma(0)|$.  However, there would be no point on $\partial A$ that was on the positive $e_1$-axis and was the same distance from the origin as $\gamma(s)$, contradicting spherical symmetry.  Since $0$ is a strict maximum point of $\gamma_{\,1}$, $\kappa(0) >0$.  Moreover, since $\gamma$ is symmetric over the $e_1$-axis, $\kappa'(0) \eeq 0.$

In addition to analyzing the curvature of $\gamma$, we will also consider the generalized mean curvature of the surface generated by $\partial A$ at a point $\gamma(s)$. 

\begin{definition}
\label{GMC}
\emph{As in \cite[Defn. 2.3]{morgan_pratelli}, we define \emph{generalized mean curvature} of a hypersurface in ${\R}^n$ with density $ f(x)  \eeq   e^{\, \psi(x)}$ by
\begin{equation}
\label{generalized mean curvature definition}
H_f  \eeq   H_0  \pl  \frac{\partial \psi}{\partial \nu} \text{,}
\end{equation}
where \( H_0 \) is the unaveraged Riemannian mean curvature and \( \nu \) is the outward unit normal vector.   If \( \psi(x)  \eeq   g(|x|) \) for some smooth function \( g \), then
\begin{equation}
\label{formula for generalized mean curvature}
H_f(x)  \eeq   H_0(x)  \pl  g'(|x|)\frac{x}{|x|} \st{\cdot} \nu(x)
\end{equation}
for any regular point $x$ on the hypersurface with $x \neeq 0$.
In ${\R}^n$ with density $r^p$, $g(r)  \eeq   \log{(r^p)}$.  Henceforth, we will denote 
\[
\frac{\partial \psi}{\partial \nu}(x)
\]
by \(H_1(x) \).  For concision, given a point \( \gamma (s) \), we refer to \( H_1( \gamma(s) ) \) as  $H_1(s)$ with analogous notation for the values of $H_0$ and $H_f$ at $\gamma(s)$.}
\end{definition}

 The following lemma of Chambers gives a useful result of spherical symmetrization.
\begin{lemma}
\label{tangent_restriction} \emph{(Tangent Restriction)} \cite[Lemma 2.6]{chambers} 
For every $s \in (0,\beta)$,   
\(
\gamma(s) \st{\cdot} \gamma \, '(s)  \leeq 0.
\)
\end{lemma}

\begin{figure}
\begin{center}

 \includegraphics{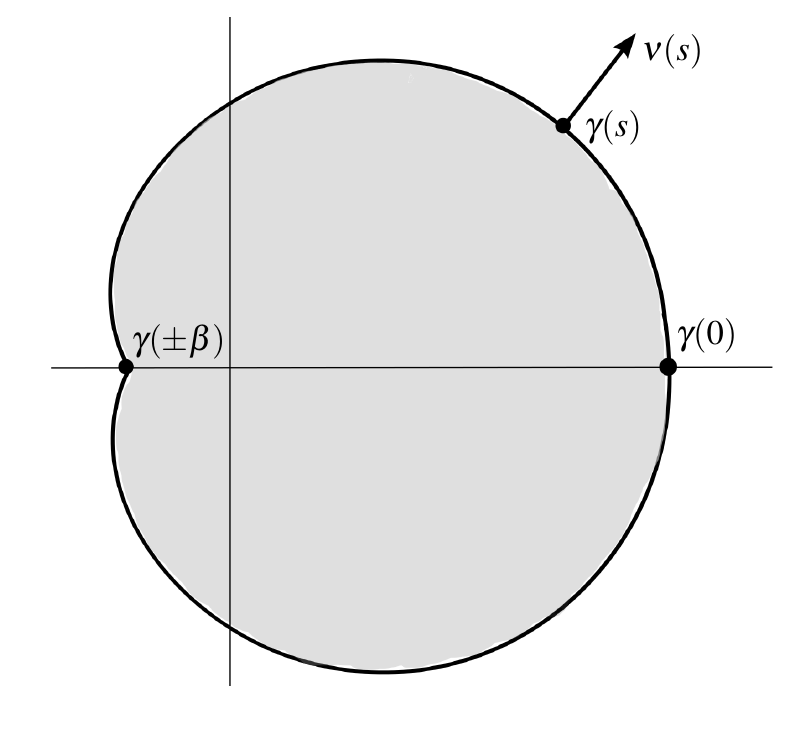}
\caption{The generating curve $\gamma$ and the outward unit normal vector at a point $\gamma(s)$ \label{fig:gamma}}		
\end{center}
\end{figure}

At each point on $\gamma$, we define a related circle that we call the canonical circle.  We show in Proposition \ref{unavg_std_curvature_formula} that the curvature of the canonical circle accounts for one of two terms in a formula for the mean curvature of the surface of revolution.

\begin{definition}
\label{canonical_circle} \cite[Defns. 3.1, 3.2]{chambers} 
\emph{Given $s \st{\in} (-\beta, \beta)$ with $s \neeq 0$, let the} canonical circle 
at $s$, \emph{denoted $C_s$, be the unique oriented circle centered on the $e_1$-axis that passes through $\gamma(s)$ and has unit tangent vector at $\gamma(s)$ equal to $\gamma \, '(s)$. 
If $\gamma \,'(s)$ is a multiple of $e_2$, then $C_s$ is an oriented vertical line. 
We define $C_0$ to be $\lim_{s \rightarrow 0} C_s$.
The regularity of the surface at $\gamma(0)$ guarantees the existence of this limit. We let $R(s)$ denote the radius of $C_s$ and let $\lambda(s)$ denote its signed curvature.  Then \( \lambda(s)  \eeq   1/R(s) \) if \(C_s \) is counterclockwise oriented, and \( \lambda(s)  \eeq   -1/R(s) \) if \(C_s \) is clockwise oriented.  
Finally, we let $F(s)$ denote the abscissa of the center of $C_s$.}
\end{definition}

The following lemma shows that spheres through the origin have constant generalized mean curvature.  We apply this result to prove Lemmas  \ref{cn_circ_gamma_circ} and \ref{kappa(0) equals lambda(0)}, which imply that $\gamma$ is a sphere through the origin, given that the curvature at the rightmost point is the same as the curvature of the circle through that point and the origin. 
\begin{proposition} 
\label{const MC}
In $\R^n$ with density $r^p$, hyperspheres through the origin have constant generalized mean curvature.
\end{proposition}

\begin{proof} Let $S$ be a hypersphere through the origin, and assume without loss of generality that $S$ can be obtained by rotating a circle $C$ in the plane about the $e_1$-axis.  It suffices to prove that  generalized mean curvature is the same at all points on \(C\).
 $H_0$ is constant on $C$ since it is constant on \( S \).  It remains to prove that $H_1$ is constant on $C$.\\
\indent
Let the center of $C$ be $(a,0)$ with $a \st{>}0$.  Then  the polar coordinates equation for $C$ is $r  \eeq   2a \cos \theta$.  
At a point $(r(\theta), \theta)$, the outward unit normal vector makes angle $2 \theta$ to the positive $e_1$-axis, and the angle between the position vector and the outward unit normal vector is $\theta$.  Supposing that \(x \) has polar coordinates $ (r, \theta)$, we have
\[
g'(|x|)\frac{x}{|x|} \st{\cdot} \nu(x)  \eeq   \frac{p}{r} \cos \theta  \eeq   \frac{p}{r} \frac{r}{2a}  \eeq   \frac{p}{2a}.
\]
 Therefore, $H_1$ is constant on $C$, as required.
\end{proof} 

\begin{remark}  
\label{r^p_only}
\emph{These computations show that the only density on $\R^2 \mn \{0\}$ ($\R^n \mn \{0\}$) for which circles (spheres) through the origin are isoperimetric is $r^p$, or a constant multiple thereof.   On a circle $C$ through the origin, parameterized by $\alpha$, the quantity $\alpha(t)/|\alpha(t)| \st{\cdot} \nu(t)$ is a constant multiple of the magnitude of the position vector.  Hence, for $H_1$ to be constant it must be the case that \( g \, '(r) \) is inversely proportional to \(r\).  This occurs only if $g(r) \eeq  \log(r^p)  \pl  c$ for some \( p \) and some constant $c$.}
\end{remark}

\begin{lemma}
\label{cn_circ_gamma_circ} 
\emph{(cf. \cite[Lemma 3.2]{chambers})}
For any point $s \in [0, \beta)$, if $C_s$ passes through the origin and $\kappa(s)  \eeq   \lambda(s)$, then $\gamma$ is a circle through the origin.
\end{lemma}

\begin{proof}
Supposing that $C_s$ is arclength parameterized, to prove that $C_s$ agrees with $\gamma$ locally, it suffices by uniqueness theorems concerning solutions of ODEs to prove that both satisfy the differential equation $H_f  \eeq   c$.  This is clearly true since the tangent vectors of the two curves agree at \( \gamma(s) \) and the generalized mean curvature of the surfaces generated by these curves is the same at \( \gamma(s) \).  To prove that $H_f  \eeq   c$ at all points on $C_s$, it suffices to show that $H_f$ is constant on $C_s$.  This follows from the computations in Proposition \ref{const MC}.  Having proved that $\gamma$ and $C_s$ coincide locally, we claim that, in fact, $\gamma$ and $C_s$ must coincide everywhere.\\
\indent 
Let $S  \eeq   \{t \in [-\beta, \beta]: \gamma([s,t)) \subset C_s \}$.  Since $\gamma$ and $C_s$ agree near $\gamma(s)$, $S$ is nonempty and therefore has a least upper bound $m$.    Letting $\alpha$ be an arclength parameterization of $C_s$, it follows by smoothness of $\alpha$ and of $\gamma$ that $m \in S$, that $C_s$ is tangent to $\gamma$ at $\gamma(m)$, and that $\kappa(m)  \eeq   \lambda(s)$. (To conclude smoothness of $\gamma$ at $m$, we are using our assumption that $m  \st{<}  \beta$.)  By an identical argument to that in the first paragraph,  there exists an open interval $I$ containing $m$ such that $\gamma(I) \subset C_s$, contradicting the fact that $m  \eeq   \sup S$. 
We conclude that $m  \eeq   \beta$.  
A similar argument shows that $\gamma$ coincides with $C_s$ on $[-\beta, s]$.
\end{proof}

\begin{remark}
\label{centered_circ_rmk}
 \emph{By radial symmetry, spheres centered at the origin also have constant generalized mean curvature.  Thus, if $C_s$ is centered at the  origin and $\kappa(s)  \eeq   \lambda(s)$, then $\gamma$ is a circle that is centered at the origin.  We use this result to obtain contradictions in several places.} 
\end{remark}

\begin{lemma}
\label{kappa(0) equals lambda(0)}
\cite[p. 12]{chambers}
We have that $\kappa(0)  \eeq   \lambda(0)$.
\end{lemma}

\begin{proof}
Showing that $\kappa(0)  \eeq   \lambda(0)$ is equivalent to showing that $F(0)  \eeq   1 \mn 1/\kappa(0)$.  If $\gamma_{\, 1}\,'(s)  \neeq 0$, then $$F(s)  \eeq   \frac{\gamma(s) \st{\cdot} \gamma\,'(s)}{\gamma_{\, 1}\,'(s)}.$$  
Since $\kappa(0)  \st{>} 0$ and \( \kappa \) is continuous at \( 0 \), there is a neighborhood of $0$ on which $\gamma_{\, 1} \, '(s)  \neeq 0$ except when $s  \eeq   0$.
By definition,

\[
F(0)  \eeq   \lim_{s \to 0}F(s)  \eeq   \lim_{s \to 0}\frac{\gamma(s) \st{\cdot} \gamma\,'\,(s)}{\gamma_{\, 1}\,'(s)}
 \eeq   1 - \frac{1}{\kappa(0)}.
\]
\end{proof}

\indent By Lemmas \ref{cn_circ_gamma_circ} and \ref{kappa(0) equals lambda(0)}, if $C_0$ is a circle  through the origin, then $\gamma$ is a circle through the origin. This means that if $F(0)  \eeq   1/2$, then $\gamma$ is a circle through the origin.  We argue by contradiction, taking cases according to whether $F(0)  \st{>} 1/2$ or $F(0) \st{<}  1/2$.  In each case, we obtain a result that contradicts spherical symmetry.  We state these results as the Right Tangent Lemma and the Left Tangent Lemma, and we devote a section to proving each.

\begin{proposition}
\label{right_tangent} \emph{(Right Tangent Lemma)} If  $F(0)  \st{>} 1/ 2$, then $\gamma_{\, 1}(\beta)  > 0$, $\lim_{s \rightarrow \beta^-}\gamma \,'(s)$ is in the fourth quadrant, and $\lim_{s \rightarrow \beta^-}\gamma \,'(s)  \neeq (0,-1)$.
\end{proposition}

\begin{proposition}
\label{left_tangent} \emph{(Left Tangent Lemma)} If $F(0)  \st{<}  1 / 2$, then $\gamma_{\, 1}(\beta) <  0$, $\lim_{s \rightarrow \beta^-}\gamma \,'(s)$ is in the third quadrant, and $\lim_{s \rightarrow \beta^-}\gamma \,'(s)  \neeq (0,-1)$.
\end{proposition}

\section{Preliminary Lemmas}
\indent This section contains results relevant to both cases.  Proposition \ref{unavg_std_curvature_formula} and Corollary \ref{generalized_curvature_ODE} give expressions for the mean curvature and generalized mean curvature at a point on the hypersurface generated by \( \gamma \) in terms of the curvature of \( \gamma \), the curvature of the canonical circle, and the normal derivative of the log of the density at that point.  We then discuss computational techniques that we use to determine how these functions (and others) vary with arclength.  Finally, Proposition \ref{curvature_comparison_theorem} is used in both cases to compare curvatures at pairs of points on the curve that are at the same height.

\begin{proposition}
\label{unavg_std_curvature_formula}
\cite[Prop. 3.1]{chambers}
Given a point $s \in [0, \beta)$, we have that 
\begin{equation}
\label{mean_curvature}
H_0(s)  \eeq   \kappa(s)  \pl  (n-2) \lambda(s).
\end{equation}
\end{proposition}

\begin{proof}
\indent We consider the principal curvatures of the surface at a point $P \eeq  \gamma(s)$. 
 We treat the case that $y  \eeq   \gamma_{\,2}(s)  \st{>} 0$ and that $\gamma\,'(s)  \neeq (0, \pm 1)$. A similar argument shows that (\ref{mean_curvature}) holds if $\gamma_{\,2}(s)  \st{<}  0$ and  $\gamma\,'(s)  \neeq (0, \pm 1)$.  We claim that there exists no interval on which $\gamma_{\,2}$ is identically $0$ or $\gamma\,'$ is vertical; then it will follow by smoothness of $\gamma$ that (\ref{mean_curvature})
holds at the remaining points.\\
\indent 
To prove the claim, recall that \( \gamma \) is smooth at \( 0 \) and that, as a consequence of spherical symmetry, $\kappa(0)  \st{>} 0$.  Thus, $\gamma_{\,2}$ cannot be identically $0$ on an interval including $0$.
On the other side of the curve, $\gamma(\beta)$ is defined to be the first point where the curve intersects the axis again, so even if a portion of the curve were a line segment along the $e_1$-axis, that segment would not be parameterized by the function $\gamma$.  The curve cannot have vertical tangent vector on an interval either.  If a portion of the curve were a vertical line segment, then this vertical line segment, when rotated, would generate a portion of a hyperplane, which would have zero mean curvature.  However, $H_1$ (the normal derivative of the log of the density) would vary as one moved up or down along the line segment, contradicting the fact that the surface has constant generalized mean curvature.\\
\indent With this technical point out of the way, we proceed in the case that $y  \eeq   \gamma_{\,2}(s)  \st{>} 0$ and that $\gamma\,'(s)  \neeq (0, \pm 1)$.  One of the principal curvatures at $P$ is the the curvature of $\gamma$ at this point.  The cross section of the surface obtained by fixing the first coordinate is an $(n-2)$-dimensional sphere of revolution.  The remaining principal curvatures of the surface are the principal curvatures of the sphere, which are equal.  Thus, to compute one of the principal curvatures of the sphere, it is sufficient to compute the second principal curvature of a $2$-dimensional surface in the $n \eeq  3$ case.  This second principal curvature is the normal curvature of a circle of revolution $C$.

By assumption that $y  \eeq   \gamma_{\,2}(s)  \st{>}0$, the curvature of the circle $C$ is $1/y$.  We let $n$ denote the inward unit normal vector to the surface and $N$ denote the normal vector to the circle of revolution.  Since \( y  \st{>} 0 \), \(C_s\) is counterclockwise oriented if and only if \(n \) is downward (i.e. \(n \) has a negative \(e_2\)-component).  Thus,
 \[
 \lambda(s)   \eeq   
 \begin{cases}
  \frac{1}{R(s)}, & n \text{ downward}\\
  \frac{-1}{R(s)}, &n \text{ upward}.
\end{cases}
 \]
 Meanwhile, by Meusnier's formula, the second principal curvature is given by 
\[
\kappa_2  \eeq   \frac{1}{y} \cos \phi,
\]
where $\phi$ is the angle between $n$ and $N$. 
 Again, since \( y  \st{>} 0 \), 
 \[
 \cos \phi  \eeq   
 \begin{cases}
  \frac{y}{R(s)}, & n \text{ downward}\\
  \frac{-y}{R(s)}, &n \text{ upward}.
\end{cases}
 \]
 
The first of these cases is depicted in Figure \ref{fig:kappa_2}.  
In both cases, the second principal curvature is the curvature of the canonical circle.\\

\indent Since one principal curvature of the surface at $\gamma(s)$ equals $\kappa(s)$ and all of the others equal $\lambda(s)$, mean curvature is given by
$$H_0(s)  \eeq   \kappa(s)  \pl  (n-2)\lambda(s).$$
\end{proof}

\begin{figure}
	\begin{center}
		\includegraphics[scale=0.9]{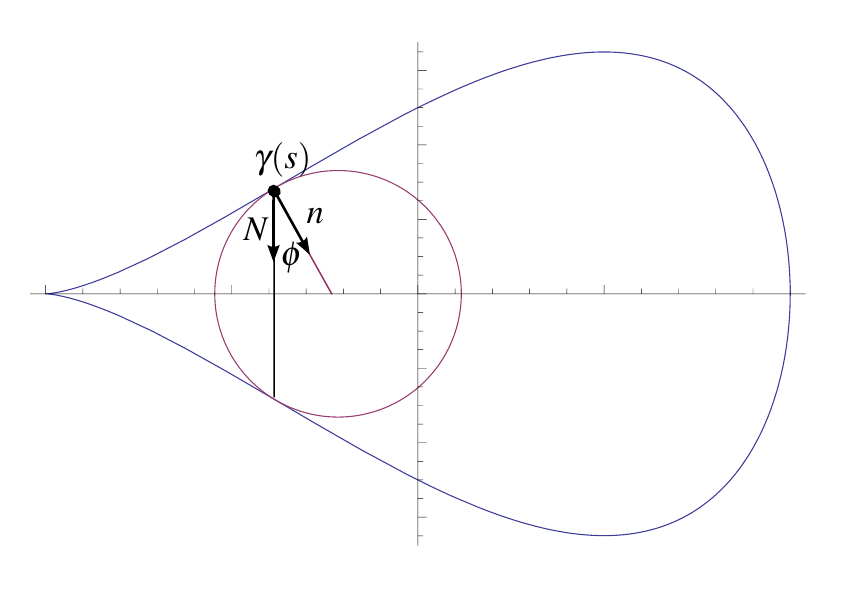}
		\includegraphics[scale=0.6]{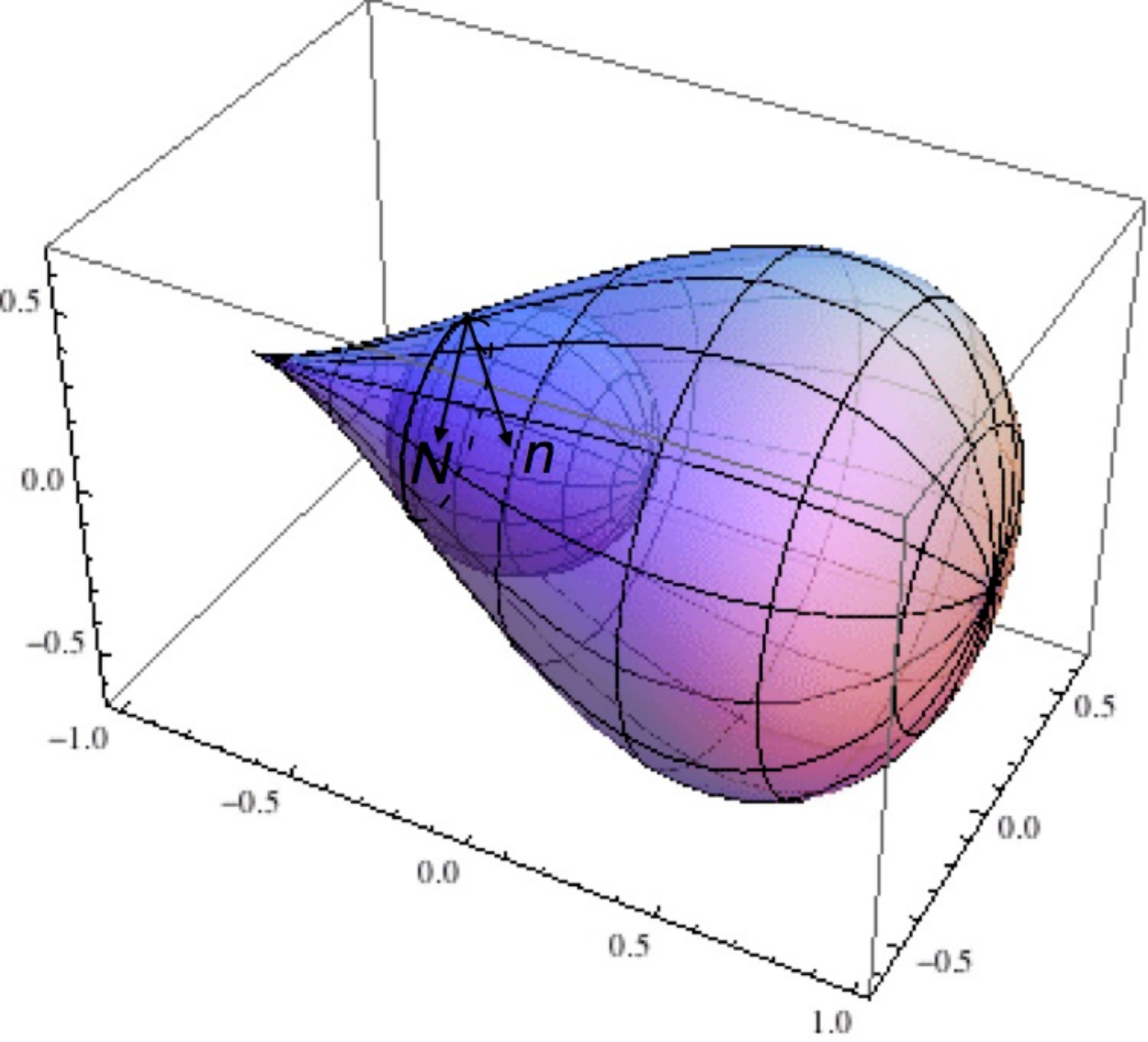}	
	    \caption{A cross section of the surface in the $xy$-plane and the inward unit normal vectors to the surface and to a circle of revolution  \label{fig:kappa_2}}
	\end{center}
\end{figure}

\begin{corollary}
\label{generalized_curvature_ODE}
Since generalized mean curvature is constant on the set of regular points of $\partial E$, there is a constant $c$ so that 
\begin{equation}
\label{H_f}
c \eeq H_f(s)  \eeq   \kappa(s)  \pl  (n-2)\lambda(s)  \pl  H_1(s).
\end{equation}
for all $s \in (-\beta, \beta)$.
\end{corollary}

\indent In the left and right cases delineated on p. 2, for any $s \in [0,\beta)$ we can analyze how $\gamma$ and related functions are instantaneously changing at $\gamma(s)$ by computing the requisite derivatives on the osculating circle to $\gamma$ at $\gamma(s)$.  A justification for this procedure will follow after we introduce some notation.
\begin{definition}
\label{as parameterized}
\emph{
Given $s$ in $(-\beta,\beta)$, let $A_s$ denote the unique oriented circle that is tangent to $\gamma$ at $\gamma(s)$ and has curvature $\kappa(s)$.  Note that if $\kappa(s)  \eeq   0$, then $A_s$ is an oriented line with direction vector $\gamma\,'(s)$.  For a fixed $s$, let $\alpha$ be an arclength parameterization of $A_s$, and let $\ti{s}$ be the point in the domain of $\alpha$ so that $\alpha(\tilde{s})  \eeq   \gamma(s)$.  For each $t$ in the domain of $\alpha$, let $\tilde{\kappa}(t)$ denote the signed curvature of $\alpha$ at $\alpha(t)$,
and let
\begin{equation*}
\ti{H_1}(t) \eeq \frac{\partial \psi}{\partial \nu}(\alpha (t)) \eeq \frac{p}{|\alpha(t)|} \frac{\alpha(t)}{|\alpha(t)|} \cdot \nu(t),
\end{equation*}
where $\nu$ is the outward unit normal vector to $\alpha$ at $\alpha(t)$. 
}
\end{definition}

Since $A_s$ is tangent to $\gamma$ at $\gamma(s)$ and has curvature $\kappa(s)$, we have 
$\alpha'(\tilde{s})  \eeq   \gamma\,'(s)$
and
$\alpha''(\tilde{s})  \eeq   \gamma\,''(s).$
Both $\tilde{\kappa}$ and $\tilde{H_1}$ are smooth functions on their domains.  

We also consider circles tangent to $\alpha$ that are centered on the $e_1$-axis, and we define analogues of the functions $F$, $R$, and $\lambda$ introduced in Definition \ref{canonical_circle}.  We use these functions to approximate their counterparts on $\gamma$ (cf. Lemma \ref{F_dominates_R_right} and Lemma \ref{Fact 10}).

\begin{definition}
\label{tilde quantities}
\emph{
Let $A_s$ and $\alpha$ be as in Definition \ref{as parameterized}.  Given $t$ in the domain of $\alpha$, let $\ti{C}_t$ denote the \emph{canonical circle to $A_s$ at $\alpha(t)$}, defined as follows: if $\alpha_2(t)  \neeq 0$, then we define $\tilde{C}_t$ to be the unique oriented circle that has its center on the $e_1$-axis and is tangent to $A_s$ at the point $\alpha(t)$.  If $\alpha_2(t)  \eeq   0$ and $\alpha'(t)  \eeq   (0, \pm 1)$, then we define $\tilde{C}_t$ to be $A_s$.  If $\alpha_2(t)  \eeq   0$ and $\alpha'(t)  \neeq (0, \pm 1)$, then $\tilde{C}_t$ is undefined.  For each $t$ so that $\tilde{C}_t$ is defined, the canonical circle is defined on a neighborhood of $t$.  We define the functions $\tilde{\lambda}$, $\tilde{R}$, and $\tilde{F}$ by letting $\tilde{\lambda}(t)$, $\tilde{R}(t)$, and $\tilde{F}(t)$ be the signed curvature of $\tilde{C}_t$, the radius of $\tilde{C}_t$, and the abscissa of the center of $\tilde{C}_t$, respectively.  Figure \ref{tilde figure} shows the osculating circle $A_s$ at a point on $\gamma$ and the canonical circle $\ti{C_t}$ at a point on $A_s$. }

\end{definition}

 \begin{figure}
\begin{center}
\includegraphics[scale=0.95]{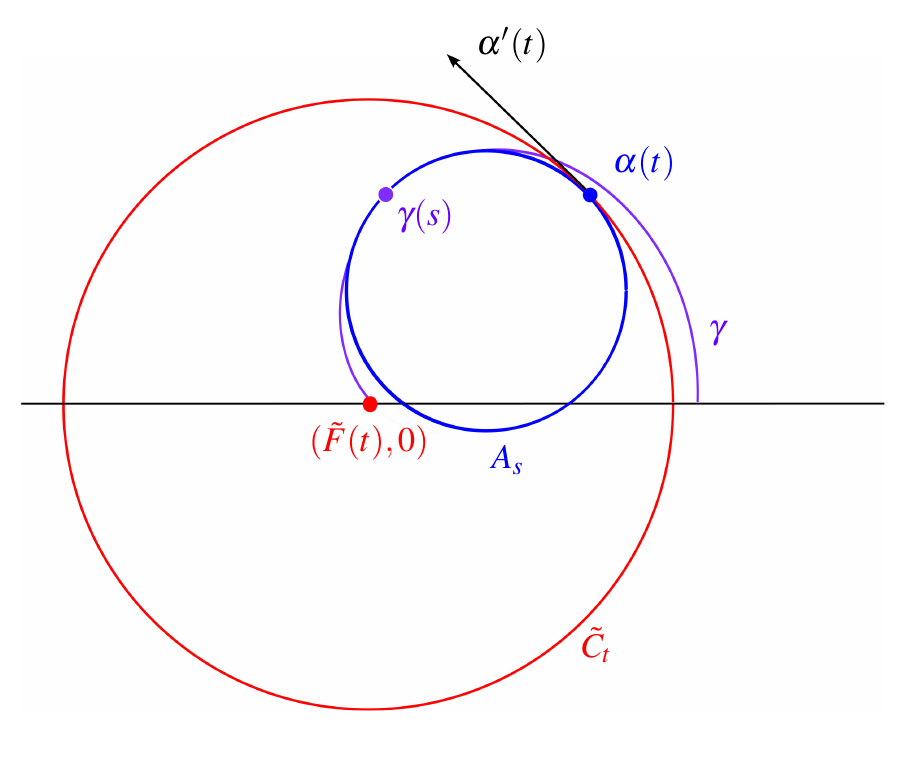}
\caption{ The osculating circle $A_s$ at a point on $\gamma$ and the canonical circle $\ti{C_t}$ at a point on $A_s$ \label{tilde figure}}
\end{center}
\end{figure}

\indent 
For a given $t$, the canonical circle $\tilde{C}_t$ depends only on $\alpha(t)$ and $\alpha'(t)$.  It follows that $\tilde{F}$, $\tilde{R}$, and $\tilde{\lambda}$ can be computed in terms of $\alpha$ and $\alpha'$ and that their derivatives depend on $\alpha$ and its first two derivatives.  In particular, since $\alpha'(\tilde{s})  \eeq   \gamma\,'(s)$ and $\alpha''(\tilde{s})  \eeq   \gamma\,''(s)$, we have 
\[
\tilde{F}'(\tilde{s})  \eeq   F'(s),
\]
 \[
 \tilde{R}'(\tilde{s})  \eeq   R'(s),
 \]
  and 
  \[
  \tilde{\lambda}'(\tilde{s})  \eeq   \lambda'(s).
  \]
  
\indent As well as analyzing these functions, we also consider the angle the tangent vector makes to the horizontal. 
\begin{definition}
\label{theta definition}
\emph{We define the function $\theta: S^1 \to (0, 2 \pi]$ by letting $\theta(v)$ be the angle in the specified interval that $v$ makes to the positive $e_1$-axis.}
\end{definition}

The next proposition of Chambers concerns two $C^2$ functions on an interval $(a,b)$.  Given $h: (a,b) \to \R_{ \geeq 0}$, we let $t_h(x)$ denote the unit tangent vector
\[
\frac{(1, h'(x))}{|(1,h'(x))|}\text{,}
\]
and we let $\kappa_h(x)$ denote the upward curvature of the graph of $h$ at $x$.
\begin{proposition}
\label{curvature_comparison_theorem}
\cite[Prop 3.8]{chambers}
Consider two $C^2$ functions $f,g : (a,b) \rightarrow \R_{ \geeq 0}$ with $b  \st{>} a$ that satisfy the following:

\begin{enumerate}[$(1)$]
\inditem $\lim_{x \rightarrow b^-} t_f(x)$ and $\lim_{x \rightarrow b^-} t_g(x)$ exist,

\inditem $\lim_{x \rightarrow b^-} f(x)$ and $\lim_{x \rightarrow b^-} g(x)$ exist,

\inditem $f'(x)  \geeq 0$ and $g'(x)  \geeq 0$ for all $x \in (a,b)$,

\inditem $\lim_{x \rightarrow b^-} f(x)  \leeq \lim_{x \rightarrow b^-} g(x)$, and $\lim_{x \rightarrow b^-} \theta(t_f(x))  \geeq \lim_{x \rightarrow b^-} \theta(t_g(x))$,

\inditem $\kappa_f(x)  \leeq \kappa_g(x)$ for all \( x \in (a,b) \). \\
\end{enumerate}
Then for every $x \in (a,b)$,
$f(x)  \leeq g(x)$,
and
$ \theta(t_f(x))  \geeq \theta(t_g(x))$.
Furthermore, if there exists a point $x_0 \in (a,b)$ such that $\kappa_f(x_0)  \st{<}  \kappa_g(x_0)$, then there is some $\phi  \st{>} 0$ such that $\phi  \leeq \theta(t_f(x)) - \theta(t_g(x))$ for all $x \in (a, x_0)$.  
\end{proposition}

\section{Proof of Right Tangent Lemma}

 \indent To prove Proposition \ref{right_tangent}, we assume that $F(0)  \st{>}  1/2$.  Then we consider two subintervals of $[0, \beta)$ that we call the upper curve and the lower curve after the objects of the same names in \cite{chambers} (see Definitions \ref{upper_curve_right} and \ref{lower_curve_right}).  We will prove that the lower curve ends in a vertical tangent at a point right of the $e_2$-axis (Lemma \ref{end_positive})  and that, past this point, curvature is positive and the tangent vector is strictly in the fourth quadrant (Lemma \ref{curving_tight_after_eta_right}).  The end behavior of the curve is similar to that of the generating curve in \cite{chambers} except that our curve must terminate right of the \( e_2\)-axis, an additional feature which allows us to achieve a contradiction to spherical symmetry without an analogue of Chambers' Second Tangent Lemma \cite[Lemma 2.5]{chambers}.  As such, many intermediate results are also similar to results in \cite{chambers} and are cross-referenced.\\
\indent Our analysis requires comparing curvatures at points of the same height on opposite sides of the curve.  Specifically, we show that the curvature at the point on the left is strictly greater than the curvature at the corresponding point on the right (Prop. \ref {lower_upper_curvature_comparison_right}).  By Corollary \ref{generalized_curvature_ODE}, it suffices to prove that $\lambda$, the canonical circle curvature, is less at the point on the left and $H_1$, the normal derivative of the log of the density, is strictly less at the point on the left.  For any $s \in [0, \beta)$, \( \gamma(s)  \neeq (0,0) \), so the normal derivative of $\log(r^p)$ at $\gamma(s)$ is given by
\begin{equation}
\label{H_1_computation}
H_1(s)  \eeq   \frac{p}{|\gamma(s)|} \frac{\gamma(s)}{|\gamma(s)|} \st{\cdot} \nu(s)  \eeq   p \frac{\gamma(s)}{|\gamma(s)|^2} \st{\cdot} \nu(s).
\end{equation}
More generally, given points $(x_1, y), (x_2,y) \in \R^2 \mn \{0\}$, and unit vectors $v_1$ and $v_2$, one can compare the quantities
\[
 \frac{(x_1,y)}{|(x_1, y)|^2} \st{\cdot} v_1^{\perp}\,\,\,\,
\text{and}\,\,\,\,
 \frac{(x_2,y)}{|(x_2,y)|^2} \st{\cdot} v_2^{\perp} \text{,}
\]
where $v_1^{\perp}$ and $v_2^{\perp}$ denote clockwise rotations of $v_1$ and $v_2$ by $\pi/2$ radians.  (In our context, $v_1$ and $v_2$ will be tangent vectors to the curve at two points, so $v_1^{\perp}$ and $v_2^{\perp}$ will be the outward unit normal vectors.)  We have discovered a set of sufficient conditions for the points $(x_1,y)$ and $(x_2,y)$ and the vectors $v_1$ and $v_2$ to satisfy the inequality
\begin{equation}
\label{H_1_comp_right_intro}
\frac{(x_1,y)}{|(x_1, y)|^2} \st{\cdot} v_1^{\perp}  \st{>}
\frac{(x_2,y)}{|(x_2,y)|^2} \st{\cdot} v_2^{\perp}.
\end{equation}

In Definition \ref{admissible_right},
we define two unit vectors $v_1$ and $v_2$ to be admissible with respect to $(x_1,y)$ and $(x_2,y)$ if they satisfy these conditions.  

\begin{definition}
\label{admissible_right}
\emph{
Consider a pair of points $(x_1,y)$ and $(x_2,y)$ with \(y  \st{>} 0\),  and a pair of unit vectors, $v_1$ and $v_2$, which lie strictly in the second and third quadrants, respectively. Let $v_1'$ denote the reflection of $v_1$ over the $e_1$-axis.  Let $C_i$ denote the canonical circle with respect to $v_i$ at $(x_i,y)$, with center $(a_i,0)$ and radius $R_i$. As depicted in Figure \ref{fig:admissible_right}, $v_1$ and $v_2$ are} admissible \emph{with respect to $(x_1,y)$ and $(x_2,y)$ if the following occur:}

\begin{enumerate}[$(1)$]
\inditem $a_1  \st{>} R_1$,
\inditem $\theta(v_2)  \geeq   \theta(v_1')$,
\inditem $x_1 - a_1  \geeq a_1 - x_2$.
\end{enumerate}

\end{definition}

\begin{figure}
	\begin{center}
	    \includegraphics[scale=0.83]{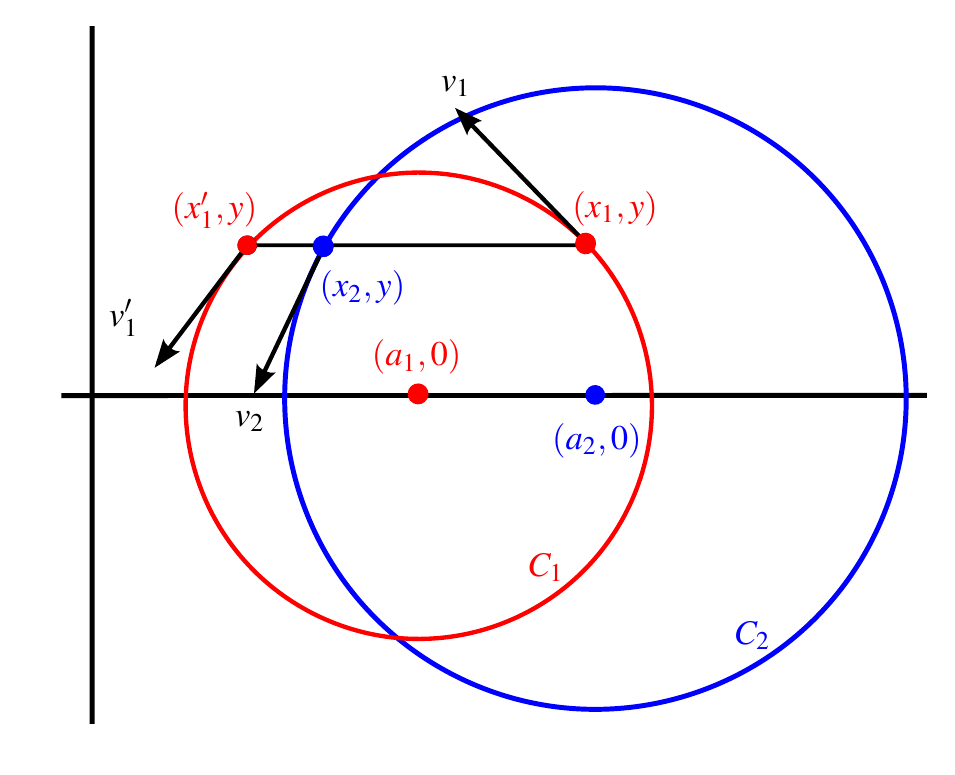}
	    \caption{The vectors $v_1$ and $v_2$ are admissible with respect to $(x_1,y)$ and $(x_2,y)$. \label{fig:admissible_right}}
	\end{center}
\end{figure}

\begin{proposition}
\label{H1_comparison_right}
Consider a pair of points $(x_1,y)$ and $(x_2,y)$ in the upper half plane with $x_1  \geeq x_2$. Let $v_1$ and $v_2$ be two unit vectors, and let $v_1^{\perp}$ and $v_2^{\perp}$ denote the clockwise rotations of these respective vectors through $\pi /2$ radians. If $v_1$ and $v_2$ are admissible with respect to $(x_1,y)$ and $(x_2,y)$, then 

\[
\frac{(x_1,y)}{|(x_1,y)|^2} \st{\cdot} v_1^{\perp}  \st{>} \frac{(x_2,y)}{|(x_2,y)|^2} \st{\cdot} v_2^{\perp}.
\]
\end{proposition}

\begin{proof}
Let $(x_1',y)$ be the reflection of $(x_1,y)$ over the vertical line $x \eeq   a_1$. It follows that $x_1'  \eeq   a_1-(x_1-a_1)$. By symmetry, $C_1$ is also the canonical circle with respect to $v_1'$ at $(x_1',y)$.   We will show that

\begin{equation}
\label{eq::H1_comparison_right_1}
\frac{(x_1,y)}{|(x_1,y)|^2} \st{\cdot} v_1^{\perp}  \st{>} \frac{(x_1',y)}{|(x_1',y)|^2} \st{\cdot} v_1'^{\perp}
\end{equation}
and that

\begin{equation}
\label{eq::H1_comparison_right_2}
\frac{(x_1',y)}{|(x_1',y)|^2} \st{\cdot} v_1'^{\perp}  \geeq \frac{(x_2,y)}{|(x_2,y)|^2} \st{\cdot} v_2^{\perp}
\end{equation}

To prove (\ref{eq::H1_comparison_right_1}), we parameterize $C_1$ by $\alpha(t)  \eeq   (a_1  \pl  R_1 \cos t , R_1 \sin t)$ for $t$ in $[0,2\pi)$. Taking $t_1 \in ( 0, \pi /2 )$ so that $\alpha(t_1)  \eeq   (x_1,y)$, we have by symmetry that $(x_1',y)  \eeq   \alpha(\pi - t_1)$. Using this parameterization to simplify the quantities in (\ref{eq::H1_comparison_right_1}), we have
\[
\frac{(x_1,y)}{|(x_1,y)|^2} \st{\cdot} v_1^{\perp} 
 \eeq   \frac{(a_1  \pl  R_1 \cos t_1, R_1 \sin t_1 )} {|\alpha(t_1)|^2} \st{\cdot} (\cos t_1, \sin t_1) \\
 \eeq   \frac{a_1 \cos t_1  \pl  R_1}{|\alpha( t_1)|^2}
\]
and 
\[
\frac{(x_1',y)}{|(x_1',y)|^2} \st{\cdot} v_1'^{ \perp}  \eeq   \frac{(a_1 - R_1 \cos t_1, R_1 \sin t_1)}{ |\alpha(\pi - t_1)|^2} \st{\cdot} (-\cos t_1, \sin t_1) \eeq 
\frac{-a_1 \cos t_1  \pl  R_1}{|\alpha(\pi - t_1)|^2}\text{,}
\]
whence
\[
\frac{(x_1,y)}{|(x_1,y)|^2} \st{\cdot} v_1^{\perp}  - \frac{(x_1',y)}{|(x_1',y)|^2} \st{\cdot} v_1'^{\perp}
 \eeq   \frac{(a_1 \cos t_1  \pl  R_1) |\alpha(\pi - t_1)|^2 - (-a_1 \cos t_1  \pl  R_1) |\alpha(t_1)|^2}{ |\alpha(t_1)|^2 |\alpha(\pi-t_1)|^2}.
\]
The denominator is positive, so we need only show that the numerator is positive to conclude that (\ref{eq::H1_comparison_right_1}) holds. A short computation reveals that 
\[
(a_1 \cos t_1  \pl  R_1) |\alpha(\pi - t_1)|^2 - (-a_1 \cos t_1  \pl  R_1) |\alpha(t_1)|^2  \eeq   2a_1  (a_1^2 -R_1^2)\cos t_1  \st{>} 0.
\]

Since $v_1$ and $v_2$ are admissible with respect to $(x_1,y)$ and $(x_2,y)$, we have that $x_2  \geeq a_1 - (x_1 - a_1)  \eeq   x_1'$. Moreover, $x_1'$ must be positive, as $a_1-(x_1 - a_1)  \st{>} a_1 - R_1  \st{>} 0$. It follows that 
\[
\frac{1}{|(x_1',y)|}  \geeq \frac{1}{|(x_2,y)|}.
\]
Therefore, to prove (\ref{eq::H1_comparison_right_2}), it suffices to show that 

\begin{equation}
\label{eq::H1_comparison_right_3}
\frac{(x_1',y)}{|(x_1',y)|} \st{\cdot} v_1'^{\perp}  \geeq \frac{(x_2,y)}{|(x_2,y)|} \st{\cdot} v_2^{\perp}.
\end{equation}

We note that the left-hand side of (\ref{eq::H1_comparison_right_3}) is
$ \cos (\theta(v_1'^{\perp}) - \theta( (x_1',y)) )$
 and the right-hand side is equal to
$\cos (\theta(v_2^{\perp}) - \theta((x_2,y)))$.
Since $v_1$ is strictly in the second quadrant, $v_2$ is strictly in the third,  and $ x_2  \geeq  x_1' \st{>}0$, it follows that
$0  \st{<}  \theta(v_1'^{\bot})-\theta((x_1',y)), \theta(v_2^{\bot})-\theta((x_2,y))  \st{<}  \pi.$
As cosine is decreasing on $(0,\pi)$, it suffices to show that 
\begin{equation}
\label{eq::H1_comparison_right_4}
\theta(v_2^{\perp}) - \theta((x_2,y))  \geeq \theta(v_1'^{\perp}) - \theta( (x_1',y)).
\end{equation}
As noted above, $x_2  \geeq x_1'$, so $\theta((x_2,y))  \leeq \theta( (x_1',y))$.  
By the admissibility of $v_1$ and $v_2$, we have that $\theta(v_2^{\perp})  \geeq  \theta(v_1'^{\perp})$. Combining these inequalities establishes (\ref{eq::H1_comparison_right_4}), completing our proof of  (\ref{eq::H1_comparison_right_2}).
\end{proof}

Having proved Proposition \ref{H1_comparison_right}, we define the upper and lower curves and prove various properties that hold on these intervals.  Our definition of the upper curve is motivated by the following observation.

\begin{lemma}
\label{kappa''(0)} \emph{(cf. \cite[Lemma 3.5]{chambers})}
Given that $F(0)  \st{>}1/2$, we have $\kappa''(0)  \st{>} 0$.
\end{lemma}

\begin{proof}
Differentiating (\ref{H_f}) and substituting $0$ into the resulting equation, we have
\[
0 \eeq \kappa''(0) \pl (n-2)\lambda''(0) \pl H_1''(0).
\]
Since $\kappa'(0) \eeq 0$, $A_0$ approximates $\gamma$ up to fourth order at $\gamma(0)$.  Thus, parameterizing $A_0$ by 
\[ \alpha(t)  \eeq    \left(a  \pl  r\cos\tor, r\sin\tor\right) 
\]
over $[-\pi r, \pi r)$,
 we have that
$\lambda''(0) \eeq \ti{\lambda}''(0)$
and 
$H_1''(0) \eeq \ti{H_1}''(0)$. 
In particular, since $\ti{\lambda}$ is constant, we have that $\lambda'(0) \eeq \ti{\lambda}'(0) \eeq 0$ and $\lambda''(0) \eeq \ti{\lambda}''(0) \eeq 0$.
(One can deduce that $\ti{\lambda}$ is constant as follows: recall that for each $t \in [-\pi r, \pi r)$, $\ti{\lambda}(t)$ denotes the curvature of $\ti{C}_t$, where $\ti{C}_t$ is defined to be the unique circle that has its center on the $e_1$-axis and is tangent to $A_0$ at the point $\alpha(t)$.  $A_0$ is a circle whose center is on the $e_1$-axis.  Thus, for each $t \in [-\pi r, \pi r)$, $\ti{C}_t \eeq A_0$.)

To prove that $\kappa''(0) > 0$,
it now suffices to prove that $\ti{H_1}''(0)  \st{<}  0$.  Since $a  \eeq   F(0)$ and $r  \eeq   R(0)$, our assumption that 
\[
F(0) \st{>} \frac{1}{2} \eeq \frac{\gamma_{\,1}(0)}{2} 
\]
is equivalent to the inequality \( a \st{>} r \).
Thus, computing $\ti{H_1}''$, we have
\[
 \ti{H_1}''(0)  \eeq    \frac{p}{|\alpha (0)|^4}\,\frac{a}{r^2}\,(r^2-a^2)  \st{<}  0.
 \]
\end{proof}

\begin{definition} \emph{(cf. \cite[Defn. 3.4]{chambers})}
\label{upper_curve_right}
\emph{
Let the} upper curve \emph{$K$ be defined as the set of all $s \in [0, \beta)$  such that for all $t$ in $[0,s]$ the following properties are satisfied:}
\begin{enumerate}[$(1)$]
\inditem  $\gamma\,'(t)$ \emph{lies in the second quadrant,}
\inditem  $\kappa(t)  \geeq \lambda(t)  \st{>} 0$.
\end{enumerate}
\end{definition}

\begin{lemma} \emph{(cf. \cite[Lemma 3.11]{chambers})}
We have that \( K \) is nonempty and that \( \sup K > 0 \).
\end{lemma}

\begin{proof}
Since $\gamma\,'(0) \eeq (0, 1)$, $\kappa(0) > 0$, and $\kappa$ is continuous, we can conclude that there exists $\rho_1 \st{>} 0$ so that $\gamma\,'(s)$ lies in the second quadrant for all $s \in [0, \rho_1]$.  Meanwhile, recall that $\kappa(0) \eeq \lambda(0) > 0$ (Prop. \ref{kappa(0) equals lambda(0)}) and that $\kappa'(0) \eeq 0$ by spherical symmetry.  As deduced in the proof of Lemma \ref{kappa''(0)}, we have that $\lambda'(0) \eeq \lambda''(0) \eeq 0$.  However, $\kappa''(0) \st{> 0}$.  It follows by taking Taylor approximations that there exists $\rho_2 \st{>} 0$ so that $\kappa(s) \geeq \lambda(s) \st{>} 0$ for all $s \in [0, \rho_2]$.
Taking \( \rho \eeq \min(\rho_1, \rho_2) \), it follows that \( [0, \rho] \subset K \).  Thus, 
\( K \) is nonempty and \( \sup K > 0 \).
\end{proof}

Having proved that \( \sup K > 0 \), we let 
\[
\delta  \eeq   \sup K.
\]

The following lemma extends our assumption that \( F(0) > R(0) \) and allows us to check the first condition of admissibility.

\begin{lemma}
\label{F_dominates_R_right}
If $s \in K$, then $F(s)  \st{>} R(s)$.
\end{lemma}

\begin{proof}
By the assumptions defining the right case, $F(0)  \st{>} R(0)$. We claim that \( F' \geeq 0 \) on \( K \) and \( R' \leeq 0 \) on $K$.  To prove so, we will use a similar argument to that in \cite[Lemma 5.3]{chambers}: for a fixed $s \in K$, let
\begin{equation}
\label{approx_circ_param}
\alpha(t) \eeq \left(a \pl r \cos \tor, \, b \pl r \sin \tor \right)
\end{equation}
be an arclength parameterization of $A_s$, and let $\ti{s}$ be the point in the domain of $\alpha$ so that $\alpha(\ti{s}) \eeq \gamma(s)$.  Since $\kappa(s) \geeq \lambda(s)$, it follows that $b \geeq 0$.  By the discussion following Definition \ref{tilde quantities}, $F'(s) \eeq \ti{F}'(\ti{s})$ and $R'(s) \eeq \ti{R}'(\ti{s})$.
Thus, we seek formulae for $\ti{F}(t)$ and $\ti{R}(t)$.  We will only consider $t$ for which $\alpha_2(t) \st{>} 0$.

Fix $t$ with $\alpha_2(t) \st{>} 0$.  As depicted in Figure \ref{FoRmulas figure},
the vector from $\alpha(t)$ to the center of $\ti{C_t}$ is in the direction of the inward unit normal vector at $\alpha(t)$.  An arclength parameterization of the line containing these points is given by 
\[
\beta(u) \eeq \alpha(t) \pl u 
\left[ \begin{array}{cc}
0 & -1 \\
1 & 0
\end{array}
\right] \alpha'(t)
\eeq \left[\begin{array}{c}
a \pl r \cos \tor \mn u \cos \tor \\
b \pl r \sin \tor \mn u \sin \tor
\end{array} \right].
\]
We let $u_0$ be the value of $u$ so that $\beta_2(u_0) \eeq 0$.  Then we have 
\[
u_0 \eeq \frac{b \pl r \sin \tor}{\sin \tor}.
\]
Since $\beta$ is an arclength parameterization, $u_0$ is the distance from $\alpha(t)$ to the center of $\ti{C}_t$, i.e.
\begin{equation}
\label{foRmula}
\ti{R}(t) \eeq u_0 \eeq \frac{b \pl r \sin \tor}{\sin \tor}.
\end{equation}
Meanwhile,
\begin{equation}
\label{Formula}
\ti{F}(t) \eeq \beta_1(u_0) \eeq a \mn b \cot \tor.
\end{equation}
Differentiating, we obtain 
\begin{equation}
\label{F'ormula}
\ti{F}'(t) \eeq \frac{b}{r} \csc^2 \tor,
\end{equation}
and
\begin{equation}
\label{foR'mula}
\ti{R}'(t) \eeq - \frac{b}{r} \csc^2 \tor \cos \tor.
\end{equation}
Since $b \geeq 0$, we have $\ti{F}'(\ti{s}) \geeq 0$.  Meanwhile, since $s \in K$, $\gamma\,'(s)$ is in the second quadrant.  Thus, $\cos(\ti{s}/r) \geeq 0$, from which it follows that $\ti{R}'(\ti{s}) \leeq 0$.  

Since $F' \geeq 0$ on $K$ and $R' \leeq 0$ on $K$, we have $F(s)  \geeq F(0)  \st{>} R(0)  \geeq R(s)$ for any $s \in K$.
\end{proof}

 \begin{figure}
\begin{center}
\includegraphics{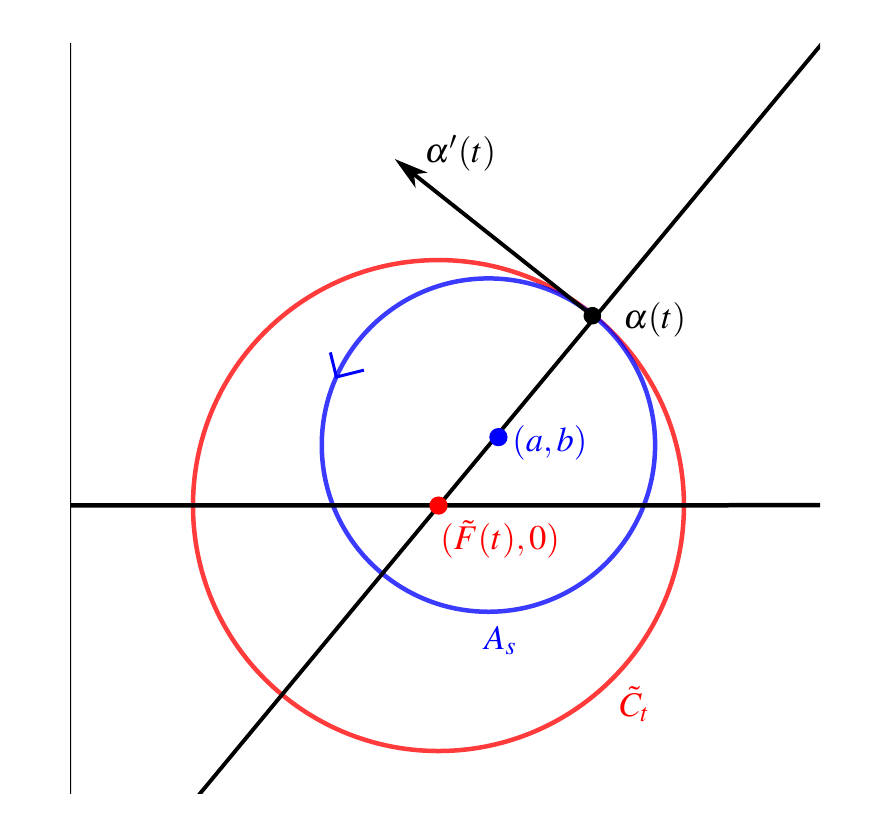}
\caption{The canonical circle at \( \alpha(t) \) and the line parameterized by $\beta$ \label{FoRmulas figure}}
\end{center}
\end{figure}

We will soon prove several properties of $\delta$, but first we require one more lemma.
\begin{lemma}
\label{app_circle_curvatures_equal}
\emph{(cf. \cite[Lemma 3.4]{chambers})}
Let $s \in (0, \delta)$. If 
 $\kappa(s)  \eeq   \lambda(s)  \st{>} 0$, then $\lambda'(s)  \eeq   0$, but $\kappa'(s)  \st{>} 0$.
\end{lemma}

\begin{proof}
Differentiating Equation (\ref{H_f}) gives $\kappa'(s)  \pl  (n-2)\lambda'(s)  \pl  H_1'(s)  \eeq   0.$  By the hypothesis that $\kappa(s)  \eeq   \lambda(s)$, we have that $A_s$  \eeq   $C_s$.  It follows that the canonical circle to $A_s$ at each point is $A_s$, so $\ti{\lambda}$ is constant. In particular,  $\lambda'(s)  \eeq   \ti{\lambda}'(\ti{s})  \eeq   0$.\\  
\indent Given this result, to prove that $\kappa'(s)  \st{>} 0$, it suffices to prove that $H_1'(s) \st{<} 0$.  Parameterizing \(A_s\) 
as in (\ref{approx_circ_param}),
we compute that 
\begin{equation}
\label{H_1_computation_right}
\tilde{H_1}'(t)  \eeq    \frac{-p(a^2  \pl  b^2-r^2)\left(-b \cos\tor  \pl  a \sin \tor \right)}{r |\alpha(t)|^4}.
\end{equation}
 Since $A_s  \eeq   C_s$, we have that $b \eeq  0$, $a \eeq  F(s)$, and $r \eeq  R(s)$.  By Lemma \ref{F_dominates_R_right},  $a \st{>}r >0$.  Finally, since \( r \sin \lrf{\ti{s}}{r} \eeq \gamma_{\,2}(s) \st{>} 0 \), we have that $ \sin \lrf{\ti{s}}{r} \st{>} 0$.  
 Thus, 
 \[
 H_1'(s)  \eeq   \ti{H_1}'(\ti{s})  \eeq   \frac{-p(a^2-r^2)\left( a \sin \left( \frac{\ti{s}}{r} \right)\right)}{r |\alpha(\ti{s})|^4} \st{<} 0. 
 \]
\end{proof}

\begin{proposition}
\label{upper_curve_structure_right} \emph{(cf. \cite[Prop. 3.12]{chambers})} The following properties of $\delta$ hold:
\begin{enumerate}[$(1)$]
\inditem $\delta  \st{<}  \beta$,
\inditem $\delta \in K$,
\inditem $\gamma_{\, 1}(\delta)  \geeq F(s)$ for any $s \in [0,\delta]$,
\inditem $\gamma_{\, 1}(\delta)  \st{>} 0$,
\inditem $\gamma \, '(\delta)  \eeq   (-1,0)$.
\end{enumerate}

\end{proposition}

\begin{proof}
The proofs of (1)-(3) are identical to their counterparts in \cite[Prop. 3.12]{chambers}.  Setting \( s \eeq 0 \) in the inequality \( \gamma_{\,1}(\delta) \geeq F(s) \), we have 
\[
\gamma_{\, 1}(\delta) \geeq F(0) \st{>} \frac{\gamma_{\,1}(0)}{2} \st{>} 0.
\]
\indent To prove that $\gamma \, ' (\delta)  \eeq   (-1,0)$, we argue by contradiction; specifically, we show that if $\gamma\,'(\delta) \neeq (-1,0)$, then there exists $\varepsilon \st{>} 0$ so that $[\delta, \delta \pl \varepsilon) \subset K$. 

Suppose that $\gamma\,'(\delta) \neeq (-1,0)$.  We have by Lemma \ref{tangent_restriction} that $\gamma\,'(\delta)$ is strictly in the second quadrant.  By continuity of $\gamma\,'$, there exists $\varepsilon_1 > 0$ so that $\gamma\,'(s)$ is in the second quadrant for all $s \in [\delta, \delta \pl \varepsilon_1)$.  Since $\delta \in K$, $\lambda(\delta) > 0$.  By continuity of $\lambda$, $\lambda > 0$ on an open interval containing $\delta$.  By reducing $\varepsilon_1$ if necessary, we can assume that $\lambda(s) > 0$ for all $s \in [\delta, \delta \pl \varepsilon_1)$.

From here, it suffices to show that there exists $\varepsilon_2 \st{>} 0$ so that $\kappa(s) \geeq \lambda(s)$ for all $s \in [\delta, \delta \pl \varepsilon_2)$.  To demonstrate the existence of such an $\varepsilon_2$, we take two cases.  Since $\delta \in K$, $\kappa(\delta) \geeq \lambda(\delta)$.  If $\kappa(\delta) \st{>} \lambda(\delta)$, then the existence of such an $\varepsilon_2$ follows by continuity of $\kappa \mn \lambda$.  Meanwhile, if $\kappa(\delta) \eeq \lambda(\delta)$, then we apply Lemma \ref{app_circle_curvatures_equal} to conclude that $\lambda'(\delta) \eeq 0$, but $\kappa\,'(\delta) > 0$.  It follows that there exists $\varepsilon_2 \st{>} 0$ so that $\kappa(s) \geeq \lambda(s)$ for all $s \in [\delta, \delta \pl \varepsilon_2)$.  
In either case, taking $\varepsilon \eeq \min(\varepsilon_1, \varepsilon_2)$ guarantees that $[\delta, \delta + \varepsilon) \subset K$, contradicting the fact that $\delta$ is an upper bound for $K$.
\end{proof}

\begin{lemma}
\label{kappa'(delta)}
We have that $\kappa'(\delta)  \st{>}0$.
\end{lemma}

\begin{proof}
Differentiating the ODE $H_f  \eeq   c$, we obtain
$\kappa'(\delta)  \pl  (n-2)\lambda'(\delta)  \pl  H_1'(\delta)  \eeq   0$.
Let $(a,b)$ be the center of $A_{\delta}$ and $r$ be its radius. Since $\kappa(\delta)  \geeq  \lambda(\delta)$, it follows that $b  \geeq  0$.  Parameterizing  $A_{\delta}$ as in (\ref{approx_circ_param}), we see that $\gamma(\delta)  \eeq   \alpha(\pi r/2)$.  Thus, $\lambda'(\delta)  \eeq   \tilde{\lambda}'(\pi r/2)$.  By inverting (\ref{foRmula}) and differentiating, we conclude that   $\tilde{\lambda}'(\pi r/2) \eeq 0$.  Since $H_1'(\delta)  \eeq   \tilde{H_1}'(\pi r/2)$, it suffices to prove that $\tilde{H_1}'(\pi r/2)  \st{<}  0$.\\
\indent Looking to (\ref{H_1_computation_right}), we claim that $a^2 \pl b^2 \st{>}r^2$.  To prove so, let $R  \eeq   R(\delta)$ be the radius of $C_{\delta}$.   As depicted in Figure \ref{horizon}, since $\gamma\,'(\delta)  \eeq   (-1,0)$, we have that $R  \eeq   r \pl b$ and $a  \eeq   F(\delta)$.  We apply Lemma \ref{F_dominates_R_right} to give
$a  \st{>} R  \eeq   r  \pl  b$.  Since both sides of the inequality $a-b \st{>}r$ are positive, we may square to give $(a-b)^2  \st{>} r^2$.  Since $b  \geeq  0$, this implies that $a^2 \pl b^2  \st{>} r^2$.  Therefore, we have that 
\[
\tilde{H_1}' \left(\frac{\pi r}{2}\right)  \eeq   \frac{-pa(a^2 \pl b^2-r^2)}{r |\alpha(\frac{\pi r}{2})|^4}  \st{<}  0.
\]
\end{proof}

\begin{figure}
\begin{center}
\includegraphics{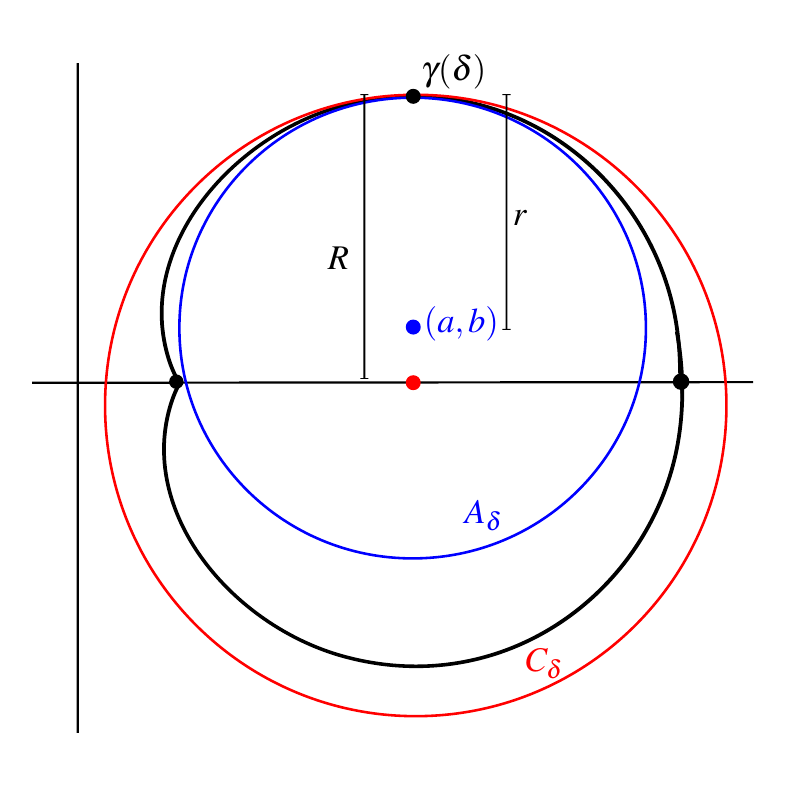}
\caption{The canonical circle $C_{\delta}$ and the osculating circle $A_{\delta}$ \label{horizon}}		
\end{center}
\end{figure}

\begin{definition}
\label{lower_curve_right} \emph{(cf. \cite[Defn. 3.5]{chambers})}
\emph{
Let the} lower curve \emph{$L$ be defined as the set of all $s$ in $[\delta,\beta)$ such that for all $t \in [\delta,s]$ the following hold:
\begin{enumerate}[$(1)$]
\inditem $\gamma\, '(t)$ is in the third quadrant with \( \gamma \, '(t)  \neeq   (-1, 0) \) if \( t  \st{>} \delta \) \text{,}
\inditem If $\overline{t}$ is the unique point in $K$ with $\gamma_{\,2}(\overline{t})  \eeq   \gamma_{\,2}(t)$, then $\kappa(\overline{t})  \leeq \kappa(t)$.
\end{enumerate}
 Since \(\gamma \, '(\delta) \eeq (-1,0) \), \( \kappa(\delta) > 0 \), and \( \kappa \, '(\delta) > 0 \), these conditions hold on an interval $[\delta, \delta  \pl  \varepsilon)$.  Thus, $L$ is nonempty and has a supremum, which we denote by $\eta$.
}
\end{definition}

By condition (1) in Definition \ref{lower_curve_right} , \( \gamma \, '(s)  \neeq   (-1, 0) \) if  $s \in (\delta, \eta)$.  Similarly,
there can be no $s_0 \in (0, \delta)$ with $\gamma \,' (s_0) \eeq (-1,0)$.  (If there were such an $s_0$, then we would have $\kappa(s_0) \geeq \lambda(s_0) \st{>} 0$.  Consequently, on an interval immediately following $s_0$, $\gamma\,'$ would be strictly in the third quadrant, contradicting the fact that $\delta$ is the least upper bound of $K$.)  Since $\gamma_{\,2}\,'$ does not vanish on $(0, \delta)$ or on $(\delta, \eta)$, we can apply the Inverse Value Theorem to define a local inverse of $\gamma_{\,2}$ over each of these intervals.

\begin{definition}
\label{local_inverses} 
\emph{
 We define $h:(\gamma_{\,2}(\eta), \gamma_{\,2}(\delta)) \to (0,\delta)$ by letting $h(y)$ be the unique $t \in (0, \delta)$ such that $\gamma_{\,2}(t) \eeq  y$. 
 Similarly, we define $k:(\gamma_{\,2}(\eta), \gamma_{\,2}(\delta)) \to (\delta,\eta)$ by letting $k(y)$ be the unique $t \in (\delta, \eta)$ such that $\gamma_{\,2}(t) \eeq  y$. 
 }
 \end{definition}
 
Using these local inverse functions, we define functions
$f,g: \,(\gamma_{\,2}(\eta), \gamma_{\,2}(\delta)) \to \R$ as follows.

\begin{definition}
\label{deFininG}
\emph{
Given $y \in (\gamma_{\,2}(\eta), \gamma_{\,2}(\delta))$, let
\begin{equation*}
f(y)  \eeq   2 \gamma_{\, 1}(\delta) - \gamma_{\, 1}(h(y)),
\end{equation*}
and let
\begin{equation*}
g(y)  \eeq   \gamma_{\, 1}(k(y)).
\end{equation*}
}
\end{definition}
The function $g$ gives the $e_1$-coordinate of a point in $\gamma(L)$ with a given $e_2$-coordinate.  If we begin with the point in $\gamma(K)$ with a given $e_2$-coordinate, then $f$ gives the $e_1$-coordinate of the reflection of this point over the line $x  \eeq   \delta$.  We can use these functions to prove two properties of the lower curve.

\begin{lemma} 
\label{lower_curve_lemma} \emph{(cf. \cite[Lemma 3.13]{chambers})}
For each $s \in [\delta, \eta)$, let $\overline{s}$ be the unique point in $K$ so that $\gamma_{\,2}(\overline{s})  \eeq   \gamma_{\,2}(s)$.  Then the following hold:
\begin{equation}
\label{gamma_one_ineq_right}
\gamma_{\, 1}(\overline{s}) - \gamma_{\, 1}(\delta)  \geeq \gamma_{\, 1}(\delta) - \gamma_{\, 1}(s)\text{,}
\end{equation}

\begin{equation}
\label{theta_ineq_right}
\theta(\gamma \, ' (s))  \geeq 2 \pi - \theta(\gamma \, '(\overline{s})).
\end{equation}
\end{lemma}

\begin{proof}
Both inequalities are trivially true if \( s  \eeq   \delta \).  Now let $s \in (\delta, \eta)$ be fixed, and let $y  \eeq   \gamma_{\,2}(s)$.  By the definition of \( L \) (Defn. \ref{lower_curve_right}), $f$ and $g$ satisfy the hypotheses of Proposition \ref{curvature_comparison_theorem}.  From the inequality $f  \leeq  g$ in Proposition \ref{curvature_comparison_theorem},  (\ref{gamma_one_ineq_right}) above is immediate.  To arrive at (\ref{theta_ineq_right}), let \( t_f(y) \) and \(t_g(y) \) denote the unit tangent vectors to the graphs of \( f \) and \( g \) at \( y\). 
Note that $\gamma(\,(\delta, \eta)\,)$ is the set $\{(g(y), y): y \in (\gamma_{\,2}  (\eta), \gamma_{\,2}  (\delta)\}$, and the reflection of $\gamma((0, \delta))$ over the line $x  \eeq   \delta$ is the set  $\{(f(y), y): y \in (\gamma_{\,2}  (\eta), \gamma_{\,2}  (\delta)\} $.
Let $y  \eeq   \gamma_{\,2}  (s)$.  Then we obtain the tangent vector $t_g(y)$ from $\gamma\,'(s)$ by rotating $\gamma\,'(s)$ clockwise through $\pi$ radians and reflecting the resulting vector in the first quadrant over the line $y  \eeq  x$.  Therefore, we have 
\[
\theta(t_g(y))  \eeq   \frac{\pi}{2} \mn \left(\theta(\gamma\,'(s)) \mn \pi \right)  \eeq   \frac{3 \pi}{2}-\theta(\gamma\,'(s)).
\]
Similarly, we obtain \( t_f(y) \) from \( \gamma \, '(\ol{s}) \) by reflecting over the line \( x  \eeq   \delta \) and reflecting over the line \( y  \eeq   x \).  Thus,
\[
\theta(t_f(y))  \eeq   \frac{\pi}{2} \mn \left(\pi - \theta(\gamma \, '(\ol{s})) \right)  \eeq   \theta(\gamma\, '(\ol{s})) \mn \frac{\pi}{2}.
\]
Substituting these results into the second inequality in Proposition \ref{curvature_comparison_theorem} completes the proof.
\end{proof}

\begin{proposition}
\label{lower_upper_curvature_comparison_right}
Let $s \in (\delta, \eta)$, and suppose that $\gamma \, '(s) \neq (0,-1)$. If $\overline{s}$ is the unique point in $K$ so that $\gamma_{\,2}(\overline{s})  \eeq   \gamma_{\,2}(s)$, then $\kappa(s)  \st{>}  \kappa(\overline{s})$. \\
\end{proposition}

\begin{proof}
Since $H_f$ is constant,
\[
\kappa(s)  \pl  (n-2) \lambda(s)  \pl  H_1(s)  \eeq   \kappa(\overline{s})  \pl  (n-2)\lambda(\overline{s})  \pl  H_1(\overline{s}).
\]
It can be shown using right triangle trigonometry and (\ref{theta_ineq_right}) from Lemma \ref{lower_curve_lemma} that \( \lambda(s)  \leeq  \lambda(\ol{s}) \).  Thus, to prove that \( \kappa(s)  \st{>} \kappa(\overline{s}) \), it suffices to prove that \( H_1(s)  \st{<}  H_1(\ol{s}) \).
We show that $\gamma \, '(\overline{s})$ and $\gamma \,'(s)$ are admissible with respect to $\gamma(\overline{s})$ and $\gamma(s)$ and then appeal to Proposition \ref{H1_comparison_right}. Since $\gamma \, '(s)$ is not equal to $(0, \mn 1)$, $\gamma \, '(s)$  lies strictly in the third quadrant. By Lemma \ref{F_dominates_R_right}, $F(\overline{s})  \st{>} R(\overline{s})$. Thus the first condition in the definition of admissibility is met.

By Lemma \ref{lower_curve_lemma}, $\theta(\gamma \, ' (s))  \geeq 2 \pi - \theta(\gamma \, '(\overline{s}))$,
satisfying the second condition of admissibility. Furthermore, by the same lemma, we have $\gamma_{\, 1}(\overline{s}) - \gamma_{\, 1}(\delta)  \geeq \gamma_{\, 1}(\delta) - \gamma_{\, 1}(s)$. By Proposition \ref{upper_curve_structure_right}, $\gamma_{\, 1}(\delta)  \geeq F(\overline{s})$, so $\gamma_{\, 1}(\overline{s}) - F(\overline{s})  \geeq F(\overline{s}) - \gamma_{\, 1}(s)$, and the final condition for admissibility is satisfied.

Because $\gamma \, '(\overline{s})$ and $\gamma \,'(s)$ are admissible with respect to $\gamma(\overline{s})$ and $\gamma(s)$, 
we conclude by Proposition \ref{H1_comparison_right} that

\[
\frac{\gamma(\overline{s})}{ |\gamma(\overline{s})|^2} \st{\cdot} \gamma \,' (\overline{s})^{\perp}  \st{>} \frac{\gamma(s)}{ |\gamma(s)|^2} \st{\cdot} \gamma \, '(s)^{\perp}.
\]
By (\ref{H_1_computation}), it follows that $H_1(s)  \st{<}  H_1(\overline{s})$, as required.  
\end{proof}

By a similar argument to that in \cite[Lemma 3.14]{chambers} along with Proposition \ref{lower_upper_curvature_comparison_right},  \( \eta  \st{<}  \beta \), \(\eta \in L \), and \( \gamma \, '(\eta) \eeq (0,-1) \).
 In addition to these properties of \( \eta \), we can also show using the curvature comparison that \(\gamma_{\, 1}(\eta)  \st{>} 0 \).  Then proving that \( \gamma_{\, 1}(\beta)  \st{>} 0 \) is a matter of showing that \( \gamma_{\, 1} \) is increasing on \( (\eta, \beta) \).  
To establish the second claim of the Right Tangent Lemma, we consider the functions \( \kappa \) and \( \gamma\,' \) on \( (\eta, \beta) \).  Lemma \ref{curving_tight_right} gives a computational result regarding \( \kappa \), whereas Lemma \ref{curving_tight_after_eta_right} extends this result as well as showing that \( \gamma\,' \) is strictly in the fourth quadrant on (\(\eta, \beta\)).

\begin{lemma}
\label{end_positive}
We have that $\gamma_{\, 1}(\eta)  \st{>} 0$.
\end{lemma}

\begin{proof}
By Lemma \ref{lower_curve_lemma},
$\gamma_{\, 1}(\overline{\eta}) - \gamma_{\, 1}(\delta)  \geeq \gamma_{\, 1}(\delta) - \gamma_{\, 1}(\eta)$.
Furthermore, $\gamma_{\, 1}(\delta)  \eeq   F(\delta)  \geeq F(\overline{\eta})$. Therefore,
$
\gamma_{\, 1}(\overline{\eta}) - F(\overline{\eta})  \geeq \gamma_{\, 1}(\overline{\eta}) - \gamma_{\, 1}(\delta) 
 \geeq \gamma_{\, 1}(\delta) - \gamma_{\, 1}(\eta) 
  \geeq F(\overline{\eta}) - \gamma_{\, 1}(\eta).
$

Finally, $\gamma_{\, 1}(\overline{\eta}) - F(\overline{\eta})  \st{<}  R(\overline{\eta})$, because $R(\ol{\eta})$ is the distance from $(F(\ol{\eta}),0)$ to $\gamma(\eta)$.  It follows that 
$
\gamma_{\, 1}(\eta)  \geeq F(\overline{\eta}) - (\gamma_{\, 1}(\overline{\eta}) - F(\overline{\eta})) 
 \st{>} F(\overline{\eta}) - R(\overline{\eta}).
$
By Lemma \ref{F_dominates_R_right}, this final expression is positive.

\end{proof}

\begin{lemma}
\label{curving_tight_right}
Let $s \in (0, \beta)$. If $\gamma_{\, 1}(s)  \geeq 0$ and $\gamma \, '(s)$ is in the fourth quadrant, then $\kappa(s)  \st{>} 0$.
\end{lemma}

\begin{proof}
Since \( \gamma \, '(s) \) is in the fourth quadrant and \( \gamma_{\,2}(s) > 0 \), \( \lambda(s) \leeq 0 \).  Since \( \gamma(s) \) is in the first quadrant and \( \nu(s) \) is in the third, \( \gamma(s) \cdot \nu(s) \leeq 0 \), which implies that \( H_1(s) \leeq 0 \).   Meanwhile, we have that \( H_f(s) \eeq H_f(0) > 0 \), because \( H_1(0) > 0 \), \( \kappa(0) > 0 \) by spherical symmetry, and \( \kappa(0) \eeq \lambda(0) \) (Proposition \ref{kappa(0) equals lambda(0)}). Hence, it must be the case that \( \kappa(s) > 0 \).
\end{proof}

\begin{lemma}
\label{curving_tight_after_eta_right}
\emph{( cf. \cite[Prop. 4.1]{chambers})}
For $s \in (\eta, \beta)$, $\gamma \, '(s)$ lies strictly in the fourth quadrant, and $\kappa(s)  \st{>} 0$.
\end{lemma}

\begin{proof}
Define $A \subset (\eta, \beta)$ so that $s \in A$ if and only if for all $t \in (\eta, s)$, $\gamma \, ' (t)$ lies strictly in the fourth quadrant and $\kappa(t)  \st{>} 0$. Note that \(A\) is nonempty because \(\kappa(\eta)  \st{>} 0\), \( \gamma\,'(\eta)  \eeq   (0, -1) \), and \( \kappa \) is continuous at \( \eta \).  Thus, \(A \) has a supremum \(\omega \).
 To prove the lemma we show that $\omega  \eeq   \beta$. \\
\indent Suppose for contradiction that $\omega  \st{<}  \beta$.  Then \(\gamma \) is smooth at \( \omega \); in particular, \(\gamma\,'(\omega) \) and \(\kappa(\omega) \) are defined.   Since $\gamma \, '(t)$ lies in the fourth quadrant for all $ t \in (\eta, \omega)$, \(\gamma\,'(\omega) \) is in the fourth quadrant.  Since $\kappa  \st{>} 0$ on \( (\eta, \omega) \),  $\gamma \, '(\omega) $ is not equal to $(0,-1)$. Furthermore, $\gamma_{\, 1}(\omega)  \st{>} 0$, as \( \gamma_{\, 1}( \eta)  \st{>} 0\) (Lemma \ref{end_positive}) and $\gamma \, '$ lies in the fourth quadrant on \( (\eta, \omega) \). 
If $\gamma \, ' (\omega)$ were equal to $(1,0)$, then we would have \( \gamma \, ' (\omega) \st{\cdot} \gamma(\omega) \eeq  \gamma_{\, 1}(\omega)  \st{>} 0 \),  contradicting the Tangent Restriction Lemma (Lemma \ref{tangent_restriction}). Thus $\gamma \, '(\omega)$ lies strictly in the fourth quadrant. By Lemma \ref{curving_tight_right}, $\kappa(\omega)  \st{>} 0$. Thus, by continuity of \( \gamma\,' \) and \(\kappa \) on \( [0, \beta) \), \(A\) could be extended past \( \omega \), contradicting the definition of \( \omega \).
\end{proof}

\begin{proof}[Proof of the Right Tangent Lemma  (Lemma \ref{right_tangent})]
It follows from Lemma \ref{curving_tight_after_eta_right} that $\gamma_{\, 1}(\beta)  \st{>} 0$, as $\gamma \, '(s)$ lies strictly in the fourth quadrant for all $s \in (\eta, \beta)$, and $\gamma_{\, 1}(\eta)  \st{>} 0$. As $\kappa  \st{>} 0$ and $\gamma \, '$ is in the fourth quadrant on $(\eta, \beta)$, the angle \( \theta(s) \) that $\gamma \, '(s) $ makes with the $e_1$-axis, measured counterclockwise in radians, 
must be a strictly increasing function on $(\eta, \beta)$ that is bounded above by $2 \pi$. Therefore, $\lim_{s \rightarrow \beta^-} \theta(s)$ exists and is in $( \theta(\eta), 2\pi]$. It follows that $\lim_{s \rightarrow \beta^-} \gamma \, ' (s)$ exists, lies in the fourth quadrant, and is not $(0,-1)$.
\end{proof}

\section{Proof of Left Tangent Lemma}
 \indent In the previous section, the key to proving the Right Tangent Lemma was to show that the curvature was greater at a point on the lower curve than at its corresponding point on the upper curve, allowing us to find $\eta  \st{<}  \beta$ where $\gamma\,'(\eta)  \eeq   (0,-1)$.  Now, in the left case (Prop. \ref{left_tangent}), we will prove the opposite inequality concerning curvatures at corresponding points, with the aim of showing that $\lim_{s \to \beta^-}\gamma\,'(s)$ is in the third quadrant and not equal to \( (0, -1) \).  This case, however, presents new obstacles. One difficulty we eliminate is the possibility that there are multiple points on the portion of $\gamma$ parameterized by $[0, \beta)$ where the tangent vector is $(-1,0)$.  In the right case, the lower curve naturally terminated at a point where the tangent vector was $(0,-1)$.  However, the goal in the left case will be to show that the lower curve does not terminate before $\beta$ (Lemma \ref{eta_equals_beta}), allowing us to apply the curvature comparison all the way up to $\beta$.
We begin with a new definition of admissibility for the left case and an analogue of Proposition \ref{H1_comparison_right}.

\begin{definition}
\label{admissible, left case}
\emph{Consider two points $(x_1,y)$ and $(x_2,y)$ and two unit vectors $v_1$ and $v_2$, strictly in the second and third quadrants, respectively.  Let $C_i$, $a_i$, $R_i$, $x_1'$, and $v_1'$ be as in Definition \ref{admissible_right}.  Finally, let $(x^*,0)$ be the unique point on the $e_1$-axis so that  $v_2$ is tangent at $(x^*,y)$ to the circle centered at the origin that passes through $(x^*,y)$.  We say that $v_1$ and $v_2$ are} admissible \emph{with respect to $(x_1,y)$ and $(x_2, y)$ if the following hold:}

\begin{enumerate}
\item $0 \st{<} a_1  \st{<} R_1$,
\item $\theta(v_2)  \leeq  \theta(v_1')$,
\item $R_2  \leeq  R_1$,
\item $x_2 \in [ x^*,x_1']$.
\end{enumerate}

\end{definition}

Figures \ref{fig: admissible_left_x_2_at_least_0} and \ref{fig: admissible_left_x_2_less_than_0} depict vectors $v_1$ and $v_2$ that are admissible with respect to $(x_1,y)$ and $(x_2,y)$ when $x_2  \geeq  0$ and when $x_2 \st{<} 0$.

\begin{figure}

	\begin{center}
		 \includegraphics{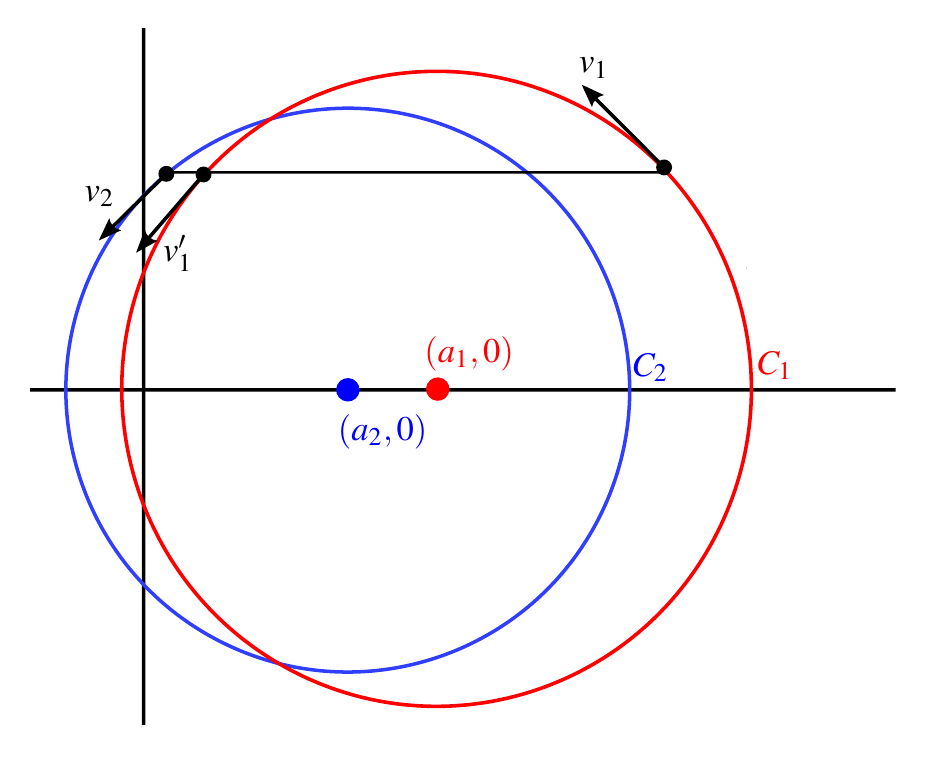}	
 
    \caption{The vectors $v_1$ and $v_2$ are admissible with respect to $(x_1,y)$ and $(x_2,y)$ with $x_2  \geeq 0$.\label{fig: admissible_left_x_2_at_least_0}}
	\end{center}
\end{figure}

\begin{figure}
	\begin{center}
		 \includegraphics{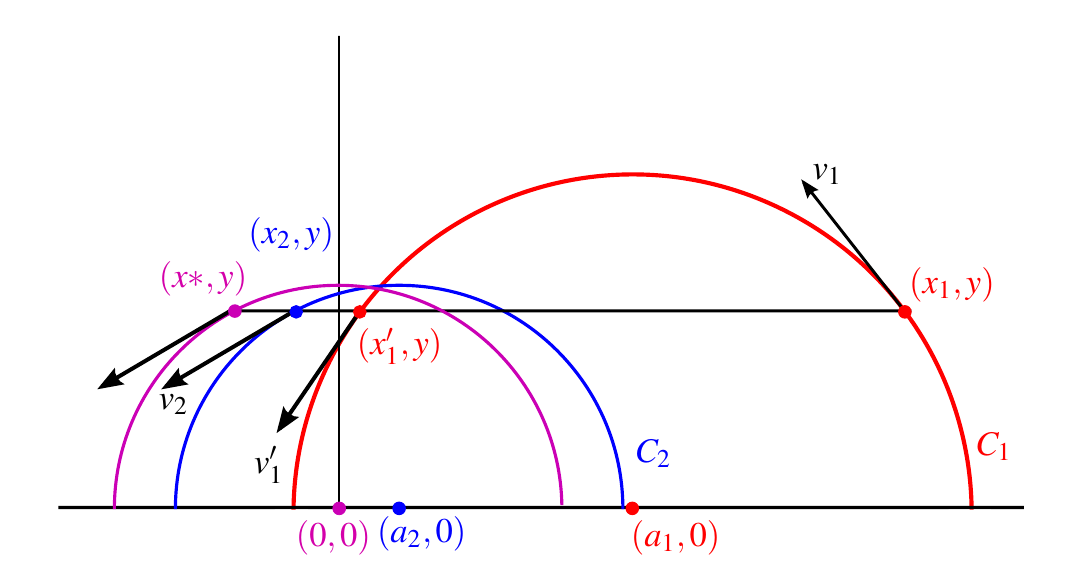}
	
	    \caption{The vectors $v_1$ and $v_2$ are admissible with respect to $(x_1,y)$ and $(x_2,y)$ with $x_2  \st{<} 0$.\label{fig: admissible_left_x_2_less_than_0}}
	\end{center}
\end{figure}

\begin{proposition} 
\label{admissibility theorem, left case}
If $v_1$ and $v_2$ are admissible with respect to $(x_1,y)$ and $(x_2,y)$, then $H_1$ is larger at $(x_2, y)$ with respect to $v_2$ than at $(x_1, y)$ with respect to $v_1$, i.e.
\[
\frac{(x_2, y)}{|(x_2,y)|^2} \st{\cdot} v_2^{\bot} \st{>} \frac{(x_1,y)}{|(x_1,y)|^2} \st{\cdot} v_1^{\bot}.
\]
\end{proposition}

\begin{proof} 
We take cases according to whether $x_2  \geeq  0$ or $x_2  \st{<}  0$.  In the case that $x_2  \geeq  0$,  $|(x_2,y)|  \leeq  |(x_1',y)|$, and the result follows by a similar argument to that in Proposition \ref{H1_comparison_right}.
 In the case case that $x_2  \st{<} 0$, we will  prove two inequalities:
\begin{align}
\frac{(x_1,y)}{|(x_1,y)|^2} \st{\cdot} v_1^{\bot} & \st{<}  \frac{(x^*,y)}{|(x^*,y)|^2} \st{\cdot} v_2^{\bot}, \label{compare:1}\\
\frac{(x^*,y)}{|(x^*,y)|^2} \st{\cdot} v_2^{\bot} & \leeq  \frac{(x_2,y)}{|(x_2,y)|^2} \st{\cdot} v_2^{\bot}.\label{compare:2}
\end{align}

\indent Beginning with (\ref{compare:1}), note that since $a_1  \st{>} 0$, we must have that $|(x_1,y)|  \st{>} R_1$.  Additionally, $|(x^*,y)|  \eeq   R_2$. Combining these observations with the inequality $R_1  \geeq  R_2$, we have 
$|(x_1,y)|  \st{>} R_1  \geeq  R_2  \eeq   |(x^*,y)|.$  It follows that
$$\frac{1}{|(x_1,y)|}  \st{<}  \frac{1}{|(x^*,y)|},$$ so proving (\ref{compare:1}) has been reduced to showing that 

$$\frac{(x_1,y)}{|(x_1,y)|} \st{\cdot} v_1^{\bot}  \leeq  \frac{(x^*,y)}{|(x^*,y)|} \st{\cdot} v_2^{\bot}  \neeq   0.$$
This inequality is immediate when we recognize that $$\frac{(x^*,y)}{|(x^*,y)|} \st{\cdot} v_2^{\bot}  \eeq   \cos(\theta(v_2^{\bot})-\theta((x^*,y)))  \eeq   \cos(0)  \eeq  1.$$
\indent To prove (\ref{compare:2}), we will rewrite the right side of the inequality using the subtraction identity for cosine.  As noted above, $\theta(v_2^{\bot})  \eeq   \theta((x^*,y))$, so 
\[
\cos(\theta(v_2^{\bot}))  \eeq   \frac{x^*}{|(x^*,y)|} \text{\indent and \indent} \sin(\theta(v_2^{\bot}))  \eeq   \frac{y}{|(x^*,y)|}.
\]
Hence, we have
\begin{equation*}
\begin{split}
\frac{(x_2,y)}{|(x_2,y)|^2} \st{\cdot} v_2^{\bot} & \eeq   \frac{1}{|(x_2,y)|}\cos(\theta(v_2^{\bot}))-\theta((x_2,y)))\\
& \eeq  \frac{1}{|(x_2,y)|} \left(\cos(\theta(v_2^{\bot}))\cos(\theta((x_2,y))) \pl \sin(v_2^{\bot})\sin(\theta((x_2,y)))\right)\\
& \eeq    \frac{1}{|(x_2,y)|} \left(\frac{x^*}{|(x^*,y)|} \frac{x_2}{|(x_2,y)|} \pl \frac{y}{|(x^*,y)|} \frac{y}{|(x_2,y)|}\right)\\
&  \eeq   \frac{1}{|(x^*,y)|} \left(\frac{x^* x_2}{|(x_2,y)|^2} \pl \frac{y^2}{|(x_2,y)|^2}\right).
\end{split}
\end{equation*}
By (4) in Definition \ref{admissible, left case} and the assumption that $x_2  \st{<}  0$, we have $x^*  \leeq  x_2  \st{<} 0$. We multiply through by $x_2$ to obtain $x^* x_2  \geeq  x_2^2  \st{>}0$.  Substituting this into the above equation, we have 
$$\frac{(x_2,y)}{|(x_2,y)|^2} \st{\cdot} v_2^{\bot}  \geeq  \frac{1}{|(x^*,y)|} \left(\frac{x_2^2 \pl y^2}{|(x_2,y)|^2}\right)  \eeq   \frac{1}{|(x^*,y)|}  \eeq   \frac{(x^*,y)}{|(x^*,y)|^2} \st{\cdot} v_2^{\bot},$$
completing the second case.
\end{proof}

\indent Before we define the upper and lower curves, we require several lemmas.  Propositions \ref{Greg's Fact 4} and \ref{F_less_than_R_persists}, which concern points where the unit tangent vector is in the second quadrant, are later used to check the conditions for admissibility.  Meanwhile, we determine some properties that hold at points on the curve with positive first coordinates.  
\begin{lemma}
\label{tangent vector points left}
Suppose that $s \in (0, \beta)$ and that $\gamma_{\, 1}(s)  \geeq  0$.  Then $\gamma_{\, 1}\,'(s)  \st{<}  0$.
\end{lemma}
\begin{proof}
Suppose for contradiction that \( \gamma_{\,1} \, '(s) \geeq 0 \).
 If \( \gamma \, '(s) \) were in the first quadrant and not equal to \( (1,0) \), this would violate the Tangent Restriction Lemma (Lemma \ref{tangent_restriction}).  If \( \gamma \, '(s) \eeq (1,0) \), then, by Lemma \ref{curving_tight_after_eta_right}, \( \kappa(s) \st{>} 0 \), which implies by continuity of \( \gamma \,' \) that there exists \(t \st{>} s\) so that \( \gamma_{\,1}(t) > 0 \), \( \gamma_{\, 2} (t) > 0 \), and \( \gamma \, '(t) \) is strictly in the first quadrant, producing the same contradiction to the Tangent Restriction Lemma.  Thus, if  \( \gamma_{\,1} \, '(s) \geeq 0 \), then \( \gamma \, '(s) \) must be in the fourth quadrant and not equal to \( (1,0) \).  However, this also yields a contradiction, because, replacing $\eta$ with $s$, we could then apply Lemmas \ref{curving_tight_after_eta_right} and \ref{curving_tight_right} to achieve the same contradiction as in the right case.   These lemmas would apply because \( \gamma_{\,1}(s) \ge 0 \).
\end{proof}

\begin{lemma} 
\label{obvious}
Suppose that $s \in (0, \beta)$ and that $\gamma_{\, 1}(s) \ge 0$.  Then $\gamma_{\, 1}(t) > 0$ for all $t \in [0,s)$.
\end{lemma}

\begin{proof}
It is clear that $\gamma_{\, 1}(0) >0$.  To prove the result on $(0, s)$, consider the set 
$C = \{t \in [0,s) : \gamma_{\, 1}(u) > 0$ for all $u \in [t,s) \}$.  By Lemma \ref{tangent vector points left}, $\gamma_{\,1}\,'(s) < 0$, so there exists $\varepsilon > 0$ such that $(s-\varepsilon,s) \st{\subset} C$.  Since $C$ is nonempty and bounded below, it has a greatest lower bound.  Let $c = \inf \, C$. 
It suffices to prove that $c = 0$.  Suppose for contradiction that $c > 0$.  By continuity of $\gamma_{\,1}$, $\gamma_{\, 1}(c) \eeq 0$;  if $\gamma_{\,1}(c)$ were positive, then we could extend $C$ farther back, whereas if it were negative, then $c$ would not be the greatest lower bound.  By Lemma \ref{tangent vector points left}, $\gamma_{\,1}\,'(c) < 0$.  It follows that there exists $\varepsilon' >0$ such that $\gamma_{\, 1} <0$ on $(c, c+ \varepsilon')$, contradicting the fact that $c$ is the greatest lower bound of $C$.
\end{proof}

\indent Now we consider the initial canonical circle $C_0$.  By spherical symmetry, $F(0)  \geeq  0$.  It must actually be the case that $F(0)  \st{>} 0$; otherwise, by Remark \ref{centered_circ_rmk}, $\gamma$ would be a circle centered at the origin, contradicting the fact that balls centered at the origin are not stable (\cite[Thm. 3.10]{rosales}).  Given this strict inequality, it follows by the computations in the proof of Lemma \ref{kappa''(0)} that 
 $\kappa''(0)  \st{<} 0$.\\
 \indent A natural next step would be to extend the inequality $F(0)  \st{<}  R(0)$.  We will eventually prove that $F(s)  \st{<}  R(s)$ for all $s$ with $\gamma_{\, 1}(s)  \geeq  0$ (Proposition \ref{F_less_than_R_persists}).  Since \( R ' \) may alternate signs, this is slightly more complicated than merely reversing the inequalities in the proof of Lemma \ref{F_dominates_R_right}. To show that the sign of $R'$ does not matter,  we define an auxiliary function that keeps track of the discrepancy between $F$ and $R$.
 
 \begin{definition}
 \label{DefininG}
 Define $G: (-\beta,\beta) \to \R$  by letting $G(s)$ be the $e_1$-coordinate of the leftmost point on $C_s$; i.e.  $G(s)  \eeq   F(s)-R(s)$.  Likewise, for a fixed $s$, if $\alpha$, $\ti{F}$, and $\ti{R}$ are as in Definition \ref{tilde quantities}, then we let $\ti{G} \eeq \ti{F} \mn \ti{R}$.
  \end{definition}
  For a given $s$, we can compute the derivatives of $\ti{F}$ and $\ti{G}$ on the approximating circle $A_s$ to prove the following lemma.
 
\begin{lemma}
\label{Fact 10}
\emph{(cf. \cite[Lemma 5.3] {chambers})} 
Let $s \in [0, \beta)$. If $\gamma_{\, 1}(s)  \geeq  0$ and $\kappa(s)  \leeq  \lambda(s)$, then $F'(s)  \leeq  0$ and $G'(s)  \leeq  0$.
\end{lemma}

\begin{proof}
We take three cases according to whether \( \kappa(s)  \st{>} 0 \), \( \kappa(s)  \eeq   0\), or \( \kappa(s)  \st{<}  0\).  If \( \kappa(s) \eeq 0 \), then 
\( A_s \) is the oriented line through \( \gamma(s) \) that has direction vector \( \gamma \,'(s) \).  By Lemma \ref{tangent vector points left}, 
$\gamma_{\,1} \, '(s) < 0$.  We parameterize \( A_s \) by \( \alpha(t) \eeq \gamma(s) \pl t \gamma \, '(s) \).  Let \( \ti{F}(t) \) denote the \(e_1\)-coordinate of the center of the canonical circle to \(A_s \) at \( \alpha(t) \), and let \( \ti{R}(t) \) denote its radius.  Then we can compute that
\[
\ti{F}'(t) \eeq \frac{1}{\gamma_{\,1} \, '(s)} < 0 \text{ \indent and \indent}
\ti{G}'(t) \eeq  \frac{1 \pl \gamma_{\, 2} \, '(s)}{\gamma_{\,1} \, '(s)}.
\]
Since \( \gamma \) is an arclength parameterization, the numerator of \( \ti{G}' \) is necessarily nonnegative.  Thus, \( \ti{G}' \leeq 0 \).\\
\indent If \( \kappa(s)  \neeq 0 \), let \( (a,b) \) be the center of \( A_s \), and let \(r \) be the radius.   If \( \kappa(s)  \st{>} 0 \), then \( b  \leeq  0 \), and we parameterize $A_s$ as in (\ref{approx_circ_param}).
By (\ref{F'ormula}) and (\ref{foR'mula}), we have 
\[
\ti{F}'(t) \eeq \frac{b}{r} \csc^2 \tor
\]
and
\[
\ti{G}'(t) \eeq \frac{b}{r} \csc^2 \tor \left(1 \pl \cos \tor \right)
\]
for all $t$ with $\alpha_2(t) > 0$.
Since $b \leeq 0$, $\ti{F}'(\ti{s}) \leeq 0$ and $\ti{G}'(\ti{s}) \leeq 0$.

Finally, if $\kappa(s) \st{<} 0$, then $b \st{>} 0$.  We now parameterize $A_s$ by 
\[
\alpha(t) \eeq \left(a \pl r \cos \tor, b \mn r \sin \tor \right).
\]
For a given $t$ with $\alpha_2(t) > 0$, the line segment from $\alpha(t)$ to the center of $\ti{C}_t$ is in the direction of the \emph{outward} unit normal vector to $A_s$ at $\alpha(t)$, as shown in Figure \ref{clockwise computations figure}.
By similar computations to those in the proof of Lemma \ref{F_dominates_R_right}, we have that 
\begin{equation*}
\ti{F} '(t) \eeq - \frac{b}{r} \csc^2 \tor,
\end{equation*}
and
\begin{equation*}
\ti{G} '(t) \eeq - \frac{b}{r} \, \csc^2 \tor \left(1 \mn \cos \tor \right).
\end{equation*}
Since $b > 0$, we conclude that $\ti{F}'(\ti{s}) \leeq 0$ and  $\ti{G}'(\ti{s}) \leeq 0$.
\end{proof}

\begin{figure}
\begin{center}
\includegraphics{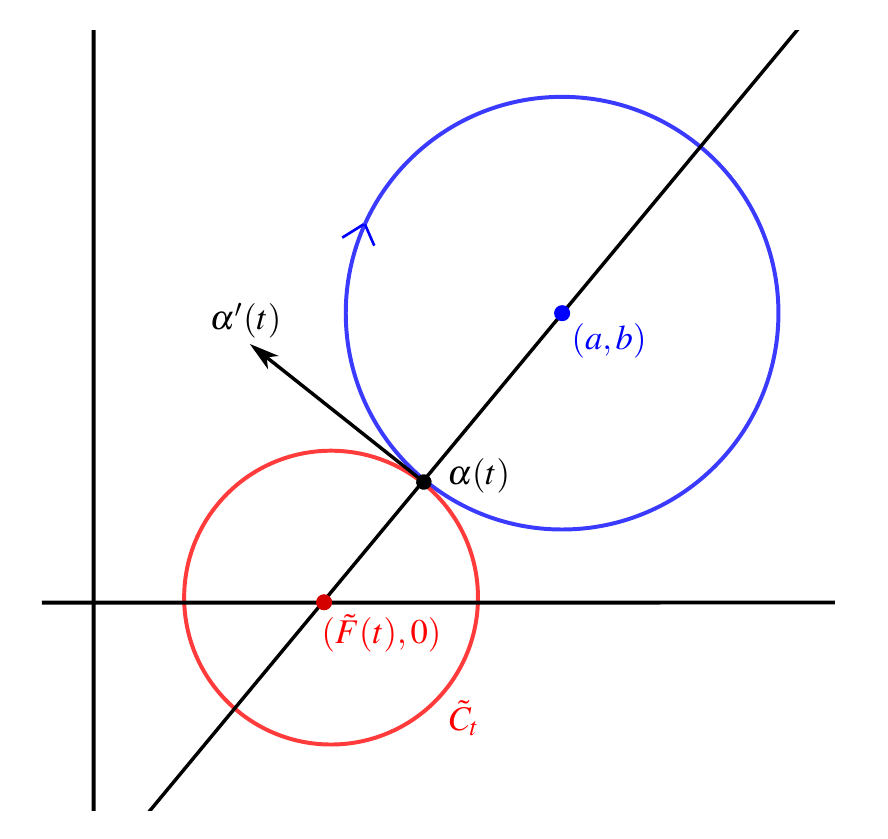}
\caption{The canonical circle at \( \alpha(t) \) in the case that \(A_s\) is oriented clockwise \label{clockwise computations figure}}
\end{center}
\end{figure}

Although we used both hypotheses of Lemma \ref{Fact 10} in the proof, it is actually the case that the first hypothesis implies the second, as we prove below.

\begin{lemma}
\label{kappa_le_lambda}
Let $s \in [0,\beta)$.  If $\gamma_{\, 1}(s)  \geeq  0$, then $\kappa(s)  \leeq  \lambda(s)$.
\end{lemma}

\begin{proof}  
Due to Lemma \ref{kappa(0) equals lambda(0)} and the fact that $\kappa'(0)  \eeq   0$, $\lambda$ and $\kappa$ are equal up to order two at $0$.  However, $\lambda''(0)  \eeq   0$, whereas $\kappa''(0)  \st{<} 0$.  Hence, there exists $t  \st{>} 0$ so that 
$\kappa  \leeq  \lambda$ on $[0,t]$.  
Let $S  \eeq   \{t \in [0, \beta): \kappa  \leeq  \lambda$ and $\gamma_{\, 1}  \geeq  0$ on $[0,t]\}$, and let $u  \eeq   \sup S$. 
Since the inequalities that define $S$ are not strict, it follows by smoothness of $\gamma$ that $u \in S$.
If  $\gamma_{\, 1}(u)  \eeq   0$, then, by Lemma \ref{obvious}, $\gamma_{\, 1}(s)  \geeq  0$ only if $s \in [0, u]$.  Thus, to prove that  $\kappa(s)  \leeq  \lambda(s)$ for all $s$ with $\gamma_{\, 1}(s)  \geeq  0$, it suffices to prove that $\gamma_{\, 1}(u) \eeq  0$.\\
\indent Suppose for contradiction that $\gamma_{\, 1}(u)  \st{>} 0$.  We will show that $u$ is not an upper bound for $S$, but, instead, that there exists $\varepsilon  \st{>} 0$ so that $[u, u \pl \varepsilon) \subset S$.  We can obviously find $\varepsilon_1  \st{>}0$ so that $\gamma_{\, 1}  \geeq  0$ on $[u, u \pl \varepsilon_1)$.  It remains to show that there exists $\varepsilon_2  \st{>}0$ so that $\kappa  \leeq  \lambda$ on $[u, u \pl \varepsilon_2)$.  The proof will be similar to that of Lemma 3.4 in [C].\\
\indent First, we can prove by contradiction 
that $\kappa(u)  \eeq   \lambda(u).$ 
Given this equation, we have that $C_u  \eeq   A_u$,  so $\lambda'(u)  \eeq   \ti{\lambda}'(\ti{u})  \eeq   0$.  Thus, to guarantee the existence of a  $\varepsilon_2  \st{>}0$ so that $\kappa  \leeq  \lambda$ on $[u, u \pl \varepsilon_2)$, it suffices to show that $\kappa'(u)  \st{<}  0$.
 Since $\kappa  \leeq  \lambda$ and $\gamma_{\, 1}  \geeq  0$ on $[0, u]$, it follows from Proposition \ref{Fact 10} that $G' \leeq  0$ on $[0,u]$.  Therefore, $G(u)  \leeq  G(0)  \st{<} 0$  by assumption that \(F(0)  \st{<}  R(0)\), and it  follows by a similar argument to that in the proof  of Lemma \ref{app_circle_curvatures_equal} that \( \kappa'(u)  \st{<}  0 \).  While the inequality \(a  \st{>} 0\) was immediate in the case that \( a  \st{>} r\), here it is more subtle.  The fact that \( a  \geeq  0 \) follows from a similar argument to that in \cite[Lemma 3.3]{chambers}.  To prove strict inequality, note that if $a  \eeq   0$, then $\gamma$ is a circle centered at the origin, which contradicts the fact that balls centered at the origin are not stable (\cite[Thm. 3.10]{rosales}).\end{proof}

We use Lemmas \ref{Fact 10} and \ref{kappa_le_lambda} to prove two propositions used in checking the conditions for admissibility (Props. \ref{Greg's Fact 4} and \ref{F_less_than_R_persists}), but first we require one additional lemma.

\begin{lemma} 
\label{Greg's Fact 5}
Suppose that $s \in (0, \beta)$ and that $\gamma\,'(s)$ is in the second quadrant.  Then $\gamma_{\, 1}(s)  \st{>}0$.
\end{lemma}

\begin{proof}
By Lemma \ref{tangent_restriction}, \( \gamma \, '(s)  \neeq   (0, 1) \).  Thus, $\gamma_{\,2}\,'(s)  \geeq  0$ and $\gamma_{\, 1}\,'(s)  \st{<} 0$.
If $\gamma_{\, 1}(s)  \st{<}  0$ or s satisfies both $\gamma_{\, 1}(s)  \eeq   0$ and $\gamma_{\,2}\,'(s)  \st{>} 0$, then we can obtain a contradiction to Lemma \ref{tangent_restriction}.  It remains to cover the case in which $\gamma_{\, 1}(s)  \eeq   0$ and $\gamma\,'(s)  \eeq   (-1,0)$.  By Proposition \ref{kappa_le_lambda}, $\kappa(s)  \leeq  \lambda(s)$.  If $\kappa(s)  \eeq   \lambda(s)$, then $\gamma$ is a circle centered at the origin, contradicting the fact that centered balls are not stable (\cite[Thm. 3.10]{rosales}).  Now, suppose that $\kappa(s)  \st{<}  \lambda(s)$.
Since $|\gamma(t)|$ is a non-increasing function of t and $C_s$ is centered at the origin,  $\gamma(t)$ must be contained in $C_s$ for $t  \geeq  s$.  However, since $\kappa(s)  \st{<}  \lambda(s)$, the curve locally leaves the disk bounded by $C_s$.
\end{proof}

\begin{proposition} 
\label{Greg's Fact 4}
Let $s \in [0, \beta)$.  If $\gamma\,'(s)$ is in the second quadrant, then $F(s)  \st{>}0$.
\end{proposition}

\begin{proof}
We know that $\gamma$ must eventually curve down and arrive at the $e_1$-axis.  Thus, there are points where $\gamma\,'$ is in the third or fourth quadrant, and, by the Intermediate Value Theorem, combined with the fact that \( \gamma \, '  \neeq (0,1) \) on $(0, \beta)$ (Lemma \ref{tangent_restriction}), there is a point $t \geeq s$ such that $\gamma\,'(t)  \eeq   (-1,0)$.  By Lemma \ref{Greg's Fact 5}, $\gamma_{\, 1}(t)  \st{>}0$; moreover,  by Lemma \ref{obvious}, $\gamma_{\, 1}  \st{>}0$ on the interval $[s,t]$.  Therefore, $F'  \leeq  0$ on $[s,t]$, from which it follows that 
$F(s)  \geeq  F(t)  \eeq   \gamma_{\, 1}(t)  \st{>}0.$
\end{proof}

\begin{proposition}
\label{F_less_than_R_persists}
If  \( \gamma\,'(s) \) is in the second quadrant, then $F(s)  \st{<}  R(s)$.
\end{proposition}

\begin{proof}
Since \( \gamma\,'(s) \) is in the second quadrant, \( \gamma_{\, 1}(s)  \st{>} 0 \).  In fact, for all \( t \in [0, s] \), $\gamma_{\, 1}(t)  \st{>} 0$, so $\kappa \leeq  \lambda$ on \( [0, s] \). Consequently, by Lemma \ref{Fact 10}, $G'  \leeq  0$ on $[0, s]$.  By hypothesis that $F(0)  \st{<}  \gamma_{\, 1}(0)/2$, we have that $G(0)  \st{<}  0$.  Therefore,
$G(s)  \leeq  G(0)  \st{<}  0$.
\end{proof}

Having proved the propositions necessary for checking the conditions of admissibility, we define the upper and lower curves and prove that the curvature at a point on the lower curve is less than the curvature at its counterpart on the upper curve (Prop. \ref{kappa comparison, left}).

\begin{definition} 
\label{upper_curve_left}
\emph{A point $s$ is in the} upper curve \emph{$K \subset (0, \beta)$ if and only if for all $t \in (0,s)$, $\gamma\,'(t)$ is strictly in the second quadrant.}
\end{definition}

Note that $K$ is nonempty because $\gamma\,'(0)  \eeq   (0,1)$ and $\kappa(0)  \st{>} 0$ (both consequences of spherical symmetry) and because \( \kappa \) is continuous at \( 0 \).  Thus, $K$ has a least upper bound $\delta$.  Since $\gamma\,'$ is strictly in the second quadrant on $(0, \delta)$, $\gamma_{\,2}(\delta)  \st{>} 0$, so $\delta  \st{<}  \beta$, from which it follows that $\gamma$ is smooth at $\delta$.  In particular, \( \gamma\,'\) is continuous at \( \delta \).  We apply the Intermediate Value Theorem, along with Lemma \ref{tangent vector points left}, to conclude that \( \gamma\,'(\delta)  \eeq   (-1, 0) \). 

\begin{definition} 
\label{lower_left}
\emph{ We define the} lower curve \emph{$L \subset [\delta, \beta)$ as follows: $s \in L$ if and only if for all $t \in [\delta, s]$, the following hold:
\begin{enumerate}
\item $\gamma\,'(t)$ is in the third quadrant, with \( \gamma \, '(t)  \neeq   (-1,0) \) if \( t  \st{>} \delta \),
\item 
\label{LL_kappa}
If $\overline{t}$ is the unique point in $K$ so that $\gamma_{\,2}(\overline{t})  \eeq   \gamma_{\,2}(t)$, then $\kappa(t)  \leeq  \kappa(\overline{t})$.
\end{enumerate}
Since $\delta \in L$,  $L$ is nonempty and therefore has a supremum, which we denote by $\eta$.} 
\end{definition}

By Proposition \ref{just_delta}, $\delta$ is the only point in $[0, \beta)$ at which the tangent vector is $(-1,0)$.  We can use this fact to prove that $\eta  \st{>} \delta$. In addition to Proposition \ref{just_delta}, our proof that $\eta  \st{>} \delta$ utilizes the following lemma, which shows that at any point on $\gamma$ where the tangent vector is \( (-1,0) \) and the curvature is $0$, the curvature has a negative derivative.  

\begin{proposition}
\label{fact_3}
Let $s \in (0, \beta)$, and suppose that $\gamma\,'(s)  \eeq   (-1,0)$. If $\kappa(s)  \geeq  0$, then $\kappa'(s)  \st{<} 0$.
\end{proposition}

\begin{proof} In the case that  $\kappa(s)  \st{>} 0$, the result follows by a similar argument to the proof of Lemma \ref{kappa'(delta)}.   Now, suppose that $\kappa(s)  \eeq   0$. The osculating circle to $\gamma$ at $\gamma(s)$ is an oriented horizontal line which we parameterize by
$\alpha(t)  \eeq   \gamma(s)  \pl  t(-1,0).$ For each $t$, $\tilde{R}(t)  \eeq   \gamma_{\,2}(s)$, so $\tilde{\lambda}(t)$ is constant; in particular, \( \lambda'(s) \eeq \ti{\lambda}'(0) \eeq 0 \).  Meanwhile, for all $t$,
\begin{equation}
\label{line_case}
\tilde{H_1}(t)  \eeq   \frac{p}{|\alpha(t)|^2}\gamma_{\,2}(s)
\end{equation}
 By Lemma \ref{Greg's Fact 5}, $\gamma_{\, 1}(s)  \st{>} 0$.  Differentiating (\ref{line_case}), we have
 \[
  H_1'(s)  \eeq   \ti{H_1}'(0)  \eeq   -2 \frac{p \, \gamma_{\,2}  (s)}{|\alpha(0)|^4} \, (-\gamma_{\, 1}(s)) \st{>} 0.
 \]
 Thus, \( \kappa'(s) \st{<} 0 \).
\end{proof}

\begin{lemma}
\label{eta_exceeds_delta_left}
Given that $\gamma([0,\beta))$ has tangent vector $(-1,0)$ only at $\delta$, we have $\eta  \st{>} \delta$.
\end{lemma}

\begin{proof}
It suffices to prove that there exist $\varepsilon_1, \varepsilon_2  \st{>}0$ so that for all \( s \in (\delta, \delta  \pl  \varepsilon_1) \), $\gamma\,'(s)$ is in the third quadrant with \( \gamma \, '(s)  \neeq   (-1,0) \), and for all \( s \in [\delta, \delta  \pl  \varepsilon_2) \),  \( \kappa(s)  \leeq  \kappa(\ol{s}) \). For the existence of such an $\varepsilon_2$, we observe that since $\gamma\,'(s)$ is in the second quadrant for all $s \in [0, \delta]$, $\kappa(\delta)  \geeq  0$.  Therefore, by Proposition \ref{fact_3}, $\kappa'(\delta)  \st{<}  0$.  To prove that there is an $\varepsilon_{\,1} \st{>}0$ as described, it suffices to prove the strict inequality $\kappa(\delta)  \st{>}0$. By Proposition \ref{fact_3}, if $\kappa(\delta)  \eeq   0$, then \( \kappa'(\delta)  \st{<}  0 \).  Hence, there exists \( q \in (\delta, \beta) \) with \( \gamma_{\,2}  (q) st{>} \gamma_{\,2}  (\delta) \).  By the Intermediate Value Theorem, applied to \( \gamma_{\,2}   \) on \( [0, \beta] \), there is a later point at the same height as \( \gamma (\delta) \).  Since \( \gamma \, '  \neeq   (0,1) \) on \( (0, \beta) \) this implies the existence of \( q \, '  \geeq  q \) with \( \gamma \,'(q \, ')  \eeq   (-1, 0) \), a contradiction.
\end{proof}

\begin{proposition}
\label{3.13 analogue, left case}
Given that $\gamma([0,\beta))$ has tangent vector $(-1,0)$ only at $\delta$, let $s \in L$ with $s  \st{>}\delta$, and let $\ol{s}$ be the unique point in $K$ so that $\gamma_{\,2}(\ol{s})  \eeq   \gamma_{\,2}(s)$.  Then the following two inequalities hold:
\begin {equation}
\label{gamma_one_ineq}
\gamma_{\, 1}(\ol{s})-\gamma_{\, 1}(\delta)  \leeq  \gamma_{\, 1}(\delta)-\gamma_{\, 1}(s),
\end{equation}

\begin{equation}
\label{theta_ineq}
\theta(\gamma\,'(s))  \leeq  2 \pi -\theta(\gamma\,'(\ol{s})).
\end {equation}

\end{proposition}

\begin{proof}
By Definitions \ref{upper_curve_left} and \ref{lower_left}, $\gamma_{\,2}\,'$ does not vanish on $(0, \delta)$ or on $(\delta, \eta)$.  Thus,
we define $k,h$ as in the proof of Lemma \ref{lower_curve_lemma}.  However, to apply Proposition \ref{curvature_comparison_theorem}, we must define $f$ and $g$ in a different way than we defined them in Definition \ref{deFininG}: now we define
 $f,g: (\gamma_{\,2}(\eta), \gamma_{\,2}(\delta)) \to \R$ by $f(y)  \eeq   \gamma_{\, 1}(k(y))$ and $g(y)  \eeq   2\gamma_{\, 1}(\delta) - \gamma_{\, 1}(h(y))$.\\
\end{proof}

\begin{proposition}
\label{kappa comparison, left}
Given that $\gamma([0,\beta))$ has tangent vector $(-1,0)$ only at $\delta$, let $s \in L$ with $s \st{>}\delta$ and \( \gamma \, '(s) \neeq (0, -1) \).  Letting $\overline{s}$ be the unique point in $K$ so that $\gamma_{\,2}(s)  \eeq   \gamma_{\,2}(\overline{s})$,  we have $\kappa(s)  \st{<}  \kappa(\overline{s})$.
\end{proposition}

\begin{proof}
Since generalized mean curvature is constant on $\gamma$, we have
\(\kappa(s)  \pl  (n-2)\lambda(s)  \pl  H_1(s)  \eeq   \kappa(\ol{s})  \pl  (n-2)\lambda(\ol{s})  \pl H_1(\ol{s}).\)
By (\ref{theta_ineq}) and right triangle trigonometry, $\lambda(s)  \geeq  \lambda(\ol{s})$.  Therefore, it suffices 
by Proposition \ref{admissibility theorem, left case} to prove that $\gamma\,'(\ol{s})$ and $\gamma\,'(s)$ are admissible with respect to $\gamma(\ol{s})$ and $\gamma(s)$.  Letting\\
 $(x_1, y)  \eeq   \gamma(\ol{s})$, $(x_2,y)  \eeq   \gamma(s)$, $v_1  \eeq   \gamma\,'(\ol{s})$, and $v_2  \eeq   \gamma\,'(s)$, we proceed to check each condition in the definition of admissibility.\\
\indent Condition (1) follows from Proposition \ref{Greg's Fact 4} and from Proposition \ref{F_less_than_R_persists}.
Recognizing that\\
 $\theta(v_1')  \eeq    2 \pi- \theta(v_1)$, we can derive condition (2) from the second inequality in Proposition \ref{3.13 analogue, left case}. Condition (3) follows by inverting the inequality $\lambda(s)  \geeq  \lambda(\ol{s})$.\\
\indent   To verify that condition (4) holds, we must show that  that $x_2  \leeq  x_1'$ and that  $x_2  \geeq  x^*$. 
The first inequality can be proved using the inequality $\gamma_{\, 1}(\ol{s})-\gamma_{\, 1}(\delta)  \leeq  \gamma_{\, 1}(\delta)-\gamma_{\, 1}(s)$ along with the fact that \( F'  \leeq  0 \) on \( (0, \delta) \) (which is a consequence of Lemmas \ref{Fact 10} and \ref{kappa_le_lambda}). To prove that $x_2  \geeq  x^*$, 
note that since $\gamma\,'(s)$ is tangent to the circle centered at the origin that passes through $(x^*,y)$, we have
$0  \eeq   (x^*,y) \st{\cdot} \gamma\,'(s)  \eeq   x^*\gamma_{\, 1}\,'(s)  \pl  y\,\gamma_{\,2}\,'(s)$.
Meanwhile, by Lemma \ref{tangent_restriction}, 
$0  \geeq  \gamma(s) \st{\cdot} \gamma\,'(s)  \eeq   \gamma_{\, 1}(s)\gamma_{\, 1}\,'(s)  \pl  \gamma_{\,2}(s)\gamma_{\,2}\,'(s)$.
Since $y \eeq   \gamma_{\,2}(s)$, it follows that $x^*\gamma_{\, 1}\,'(s)  \geeq  \gamma_{\, 1}(s)\gamma_{\, 1}\,'(s)$.  Dividing through by $\gamma_{\, 1}\,'(s)$ gives $x^*  \leeq  \gamma_{\, 1}(s)$.
\end{proof}

\begin{proposition}
\label{eta_equals_beta}
If there exists no $s  \neeq \delta$ so that $\gamma\,'(s)  \eeq   (-1,0)$, then $\eta  \eeq   \beta$.
\end{proposition}

\begin{proof}
Suppose for contradiction that $\eta  \st{<}  \beta$. We will show that $\eta$ is not an upper bound for $L$, but instead, that $L$ can be extended.   Recall that by defining a local inverse function \( h: (\gamma_{\,2}  (\eta), \gamma_{\,2}  (\delta)) \to (0, \delta) \) as on p. 13, we can explicitly write \( \ol{s}  \eeq   h(\gamma_{\,2}  (s)) \).  By continuity of $\gamma\,'$, of $\kappa$, and of $\kappa \circ h \circ \gamma_{\,2}$, \( \gamma \, '(\eta) \) is in the third quadrant, and \( \kappa(\eta) \leeq \kappa(\ol{\eta}) \).  To show that \( \eta \in L \), we need 
only show that \( \gamma \, '(\eta) \neeq (-1,0) \).  If $\gamma\,'(\eta)  \eeq   (-1, 0)$, this would contradict the fact that $\gamma$ does not have multiple horizontal tangents.  Meanwhile,  if $\gamma\,'(\eta)  \eeq   (0, -1)$, this would contradict (\ref{theta_ineq}), which holds at $\eta$ by continuity of $\gamma\,$ on $(0, \beta)$ and by our assumption that $\eta  \st{<}  \beta$.  Thus, \( \gamma \, '(\eta) \) is strictly in the third quadrant.  Finally, by an identical argument to that in Proposition \ref{kappa comparison, left}, $\kappa(\eta)  \st{<}  \kappa(\ol{\eta})$.  
\end{proof}

Having shown that $\eta  \eeq   \beta$, we are near to proving Lemma \ref{left_tangent} with the assumption that \( \delta \) is the only point at which \( \gamma \,' \) equals \( (-1, 0) \).  First, we show that  $\gamma_{\, 1}(\beta)  \st{<}  0$.  In order to discuss $\lim_{s \to \beta^-}\gamma\,'(s)$, we must first show that the limit exists.  For this, we prove in Proposition \ref{kappa_negative} that $\kappa$ is eventually negative.  The proof requires Proposition \ref{end_left} as well as a lemma giving a bound on $\gamma_{\, 1}\,'$ (Lemma \ref{gamma_one_prime_bound}).

\begin{proposition}
\label{end_left}
Given that $F(0)  \st{<}  1/2$ and that $\gamma([0,\beta))$ has tangent vector $(-1,0)$ only at $\delta$, we have that $\gamma_{\, 1}(\beta)  \st{<}  0$.
\end{proposition}

\begin{proof}
To prove that $\gamma_{\, 1}(\beta)  \st{<}  0$, we begin with Proposition \ref{3.13 analogue, left case}, which states that if $s \in L$, and $\ol{s}$ is the corresponding point in $K$ such that $\gamma_{\,2}(s)  \eeq   \gamma_{\,2}(\ol{s})$, then $\gamma_{\, 1}(\ol{s}) - \gamma_{\, 1}(\delta)  \leeq  \gamma_{\, 1}(\delta) - \gamma_{\, 1}(s)$.  Since $\beta  \eeq   \sup \, L$ and $\gamma$ is continuous at $\beta$, this inequality also holds for $s  \eeq  \beta$.  Noting that $\ol{\beta}  \eeq   0$, we have $\gamma_{\, 1}(0) - \gamma_{\, 1}(\delta)  \leeq  \gamma_{\, 1}(\delta) - \gamma_{\, 1}(\beta) $.  Since $\gamma\,'(\delta)  \eeq   (-1,0)$, $\gamma_{\, 1}(\delta)  \eeq   F(\delta)$.  In turn, since \(F\) is non-increasing on the upper curve, $F(\delta)  \leeq  F(0)$.  Consequently,  
$\gamma_{\, 1}(0) - F(0)  \leeq  \gamma_{\, 1}(0) - \gamma_{\, 1}(\delta)  \leeq  \gamma_{\, 1}(\delta) - \gamma_{\, 1}(\beta)  \leeq  F(0) - \gamma_{\, 1}(\beta).$
Rearranging gives
$\gamma_{\, 1}(\beta)  \leeq  2F(0) - \gamma_{\, 1}(0)  \st{<} 0$.
\end{proof}

\begin{lemma}
\label{gamma_one_prime_bound}
Given that $\gamma([0,\beta))$ has tangent vector $(-1,0)$ only at $\delta$, there exists $\xi  \st{>} 0$ so that $\gamma_{\, 1}\,'(s)  \leeq  -\xi$ for all $s \in L$.
\end{lemma}

\begin{proof}
It suffices to prove that  there exists \( \tau  \st{>} 0 \) such that \( \theta(\gamma\,'(s)) \st{<} 3\pi/2 - \tau \) for all \( s \) in \( L \).  Since \( \gamma \, '( \delta)  \eeq   (-1,0) \) and \( \gamma \, ' \) is continuous on \( [0, \beta) \), there exists \( s_0 \st{>} \delta \) such that \( \theta \circ \gamma \, ' \st{<}  5\pi/4 \) on \( [ \delta, s_0] \).  By Proposition \ref{kappa comparison, left}, \( \kappa(s_0) \st{<} \kappa(\ol{s_0}) \).  Letting \( y_0  \eeq   \gamma_{\,2}(s_0) \), we have that the upward curvatures of the graphs of the functions \( f \) and \( g\) defined in Proposition \ref{3.13 analogue, left case} satisfy \( \kappa_f(y_0) \st{<} \kappa_g(y_0) \).  Therefore, by Proposition \ref{curvature_comparison_theorem}, 
there exists  \( \phi  \st{>} 0 \) such that
\begin{equation}
\label{phi_left}
\theta(t_f(y))  \geeq  \theta(t_g(y))  \pl  \phi
\end{equation}
for all \( y \in (0,y_0) \). 
Take \( \tau  \eeq   \min \{ \phi, \pi/4 \} \).  If \( s \in [\delta, s_0] \), then \( \theta( \gamma \, '(s))  \st{<}  5 \pi/4  \eeq   3 \pi/2 \mn \pi/4   \leeq 3\pi/2 \mn \tau. \)
If \( s \in (s_0, \beta) \), let \( y  \eeq   \gamma_{\,2}  (s) \).  Then  \( \theta(t_f(y))  \eeq   3 \pi/2-\theta(\gamma\,'(s)) \), and \( \theta(t_g(y))  \eeq   \theta(\gamma\,'(\ol{s})) - \pi/2 \).  Substituting into (\ref{phi_left}), we obtain
\( \theta(\gamma\,'(s))  \leeq   2 \pi-\theta(\gamma\,'(\ol{s})) - \phi \).
Finally, since \( y  \st{>} 0 \), it follows by Lemma \ref{tangent vector points left} that \( \theta(\gamma\,'(\ol{s}))  \st{>} \pi/2 \).  Therefore, \( \theta(\gamma\,'(s))  \st{<}   3 \pi/2 - \phi   \leeq 3\pi/2 \mn \tau \).
\end{proof}

\begin{proposition}
\label{kappa_negative}
Given that $\gamma([0,\beta))$ has tangent vector $(-1,0)$ only at $\delta$, there exists $\varepsilon  \st{>} 0$ such that $\kappa  \st{<}  0$ on $(\beta - \varepsilon, \beta)$.
\end{proposition}

\begin{proof}
We show that for $s$ close to $\beta$, we can make $(n-2)\lambda(s)  \pl  H_1(s)$ larger than $c$, the constant of the differential equation $H_f  \eeq   c$.  First, we show that by taking $s$ sufficiently close to $\beta$, we can make $\lambda(s)$ large.
The radius of the canonical circle at $\gamma(s)$ satisfies
\begin{equation*}
\begin{split}
R(s)^2  \eeq   (\gamma_{\, 1}(s) - F(s))^2  \pl  \gamma_{\,2}(s)^2  \eeq    \left(\gamma_{\, 1}(s) - \frac{\gamma(s) \st{\cdot} \gamma\,'(s)}{\gamma_{\, 1}\,'(s)} \right)^2  \pl  \gamma_{\,2}(s)^2  \eeq   \frac{\gamma_{\,2}(s)^2}{\gamma_{\, 1}\,'(s)^2}.
\end{split}
\end{equation*}
Since $\gamma_{\, 1}\,'(s)  \st{<}  0$, we have 
$\displaystyle R(s)  \eeq   \frac{\gamma_{\,2}(s)}{-\gamma_{\, 1}\,'(s)}$\,\, and\,\, $\displaystyle \lambda(s)  \eeq   \frac{-\gamma_{\, 1}\,'(s)}{\gamma_{\,2}(s)}$.\\
By Lemma \ref{gamma_one_prime_bound}, 
$\lambda(s)  \geeq  \xi/\gamma_{\,2}(s)$.  Since $\gamma_{\,2}(\beta)  \eeq   0$ and $\gamma$ is continuous, there exists $\varepsilon_1  \st{>} 0$ so that if $s \in (\beta - \varepsilon_1, \beta)$, then $\gamma_{\,2}(s)  \st{<}  \xi/c$.  Consequently, for all $s \in (\beta - \varepsilon_1, \beta)$, $\lambda(s)  \st{>} c$.\\
\indent Now we will show that for $s$ sufficiently large, $H_1(s)$ is positive.
By Proposition \ref{end_left} and continuity of $\gamma$, there exists $\varepsilon_2  \st{>} 0$ such that $\gamma_{\, 1} \st{<} 0$ on $(\beta-\varepsilon_2, \beta)$.  For any $s$ in this interval, $\gamma(s)$ and $\nu(s)$ are both strictly in the second quadrant, so $H_1(s)  \st{>}0$.\\
\indent Set $\varepsilon  \eeq   \min \{\varepsilon_1, \varepsilon_2\}$, and suppose that $s \in (\beta-\varepsilon, \beta)$.  By our observations above and our assumption that $n  \geeq  3$, we have  $(n-2)\lambda(s)  \pl  H_1(s)  \st{>} (n-2)\lambda(s)  \geeq  \lambda(s)  \st{>} c$.  Therefore, $\kappa(s)$ must be less than 0 to compensate.
\end{proof}

\begin{proof}[Proof of the Left Tangent Lemma (Lemma \ref{left_tangent})]
By Proposition \ref{end_left}, \( \gamma_{\, 1}(\beta) \st{<} 0 \).  By Proposition \ref{kappa_negative}, there exists $\varepsilon \st{>}0$ such that $\kappa \st{<} 0$ on $(\beta - \varepsilon, \beta)$.  On this interval, $\theta \circ \gamma\,'$ is a decreasing function of $s$.  Since $\theta \circ \gamma\,'$ is decreasing and bounded below by $\pi$, $\lim_{s \to \beta^-} \gamma\,'(s)$ exists.  Moreover, since $\gamma\,'$ is strictly in the third quadrant on $(\delta, \eta)$ and $\theta \circ \gamma\,'$ is decreasing on $(\beta - \varepsilon, \beta)$, $\lim_{s \to \beta^-} \gamma\,'(s)$ is in the third quadrant but not equal to $(0,-1)$.  
\end{proof}

\subsection{Proof That There is Only One Horizontal Tangent}
Finally, we supply a proof of the result used from Proposition \ref{eta_exceeds_delta_left} onward that \( \delta \) is the only point in \( [0, \beta) \) with tangent vector \( (-1,0) \).  It is expedient to  consider the sets
 \( T \eeq   \{s \in [0, \beta): \gamma\,'(s)  \eeq   (-1,0)$ and $\kappa(s)  \geeq  0\}$ and $U  \eeq   \{s \in [0, \beta): \gamma\,'(s)  \eeq   (-1,0) \}$.  Consider the supremum \( \delta \) of the upper curve \( K \) (Defn. \ref{upper_curve_left}).  In the proof of Lemma \ref{eta_exceeds_delta_left}, by assuming that \( \delta \) was the only point in \( [0, \beta) \) where the tangent vector was \( (-1,0) \) (the fact that we are about to prove), we could show that \( \kappa(\delta) > 0 \).  However, even without this assumption, it must be the case that \( \kappa (\delta) \ge 0 \), because \( \gamma \, ' \) is strictly in the second quadrant on \( (0,\delta) \) (cf. proof of Lemma \ref{eta_exceeds_delta_left}).  Thus, \( \delta \in T \).  Since \( T \) is nonempty, it has a least upper bound \( v\).  

\begin{lemma}
\label{sup_T}
The supremum of \( T \) satisfies the following:
\begin{enumerate}
\item $v  \st{<}  \beta$,
\item $v$ is the maximum element of U,
\item \( \kappa(v) > 0 \).
\end{enumerate}
\end{lemma}

\begin{proof}
To prove that  \( v < \beta \), it suffices to show that there there exists $\varepsilon_0  \st{>} 0$ so that if $s \in (\beta - \varepsilon_0, \beta)$ and $\gamma\,'(s)  \eeq   (-1,0)$, then $\kappa(s)  \st{<}  0$.  To achieve this result, we consider the ODE $H_f \eeq  c$. 
We know that the constant $c$ is positive, because $H_1(0)  \eeq   p$,  $\kappa(0)  \st{>} 0$, and $\lambda(0)  \eeq   \kappa(0)$ by Proposition \ref{kappa(0) equals lambda(0)}.\\
\indent Since $\gamma_{\,2}(\beta)  \eeq   0$ and the curve is continuous at $\beta$, there exists $\varepsilon_0  \st{>} 0$ so that for any s in  $(\beta-\varepsilon_0, \beta)$, we have $\gamma_{\,2}(s)  \st{<}  1/c$.\,  Let $s \in  (\beta-\varepsilon_0, \beta)$ and suppose that $\gamma\,'(s)  \eeq   (-1,0)$.  Then $\lambda(s)  \eeq   1/\gamma_{\,2}(s) \st{>}c.$  Meanwhile, the outward unit normal at $s$ is $\nu(s)  \eeq   (0,1)$, so \[
H_1(s)  \eeq   \frac{p}{|\gamma(s)|^2}\,(\gamma_{\, 1}(s), \gamma_{\,2}(s)) \st{\cdot} (0, 1)  \eeq   \frac{p}{|\gamma(s)|^2}\, \gamma_{\,2}(s)  \st{>} 0.
\]
  Given that $n  \geeq  3$, we have that $(n-2)\lambda(s)  \pl  H_1(s)  \geeq  \lambda(s)  \pl  H_1(s)  \st{>}c$, which means that $\kappa(s)$ must be negative to compensate.\\
\indent Given that $v \st{<} \beta$, it can be shown by continuity of $\gamma\,'$ and $\kappa$ on $(0, \beta)$ that $\gamma\,'(v)  \eeq   (-1,0)$ and that $\kappa(v)  \geeq  0$. Since \( \gamma \, '(v) \eeq (-1,0) \), \( v \in U \eeq   \{s \in [0, \beta): \gamma\,'(s)  \eeq   (-1,0) \} \). 
We claim that $v$ is the largest point in $U$.  By definition of \( T\), there exists no \( s \st{>} v \) so that \( \gamma \, '(s) \eeq (-1,0) \) and \( \kappa(s) \geeq 0 \).  Meanwhile, if there were an \( s \st{>} v \) so that \( \gamma \, '(s) \eeq (-1,0) \) and \( \kappa(s) \st{<} 0 \), then \( s \) would be a local minimum point of \( \gamma_{\,2} \).  Since \( \gamma_{\,2}(\beta) \eeq 0 \), there would exist \( t \st{>} s \) so that \( t \) was a local maximum point of \( \gamma_{\, 2} \).  Since \( \gamma \, ' \neeq (0,1) \) on  \( (0, \beta) \) (Lemma \ref{tangent_restriction}), \( \gamma \, '(t) \neeq (1,0) \).  Thus, it must be the case that  \( \gamma \, '(t) \eeq (-1,0) \), contradicting the fact that \(v \eeq \sup \, T \).  We conclude that \( v \) is the maximum element of \( U \).  Again, since  \( \gamma \, ' \neeq (0,1) \) on  \( (0, \beta) \), this means that \( \gamma_{\,2} \, ' < 0 \) on \( (v ,\beta) \). \\
 \indent Finally, to prove that \( \kappa(v) > 0 \), suppose for contradiction that \( \kappa(v) \eeq 0 \).  By  Lemma \ref{fact_3}, there exists \( \varepsilon > 0 \) so that  \( \kappa \st{<} 0 \) on \( (v, v \pl \varepsilon) \).  Since \( \gamma \,'(v) \eeq (-1,0) \), this would imply the existence of an interval following \( v \) on which the tangent vector was strictly in the second quadrant, contradicting the fact that \( \gamma_{\,2} \,' < 0 \) on \( (v, \beta) \) (cf. proof of Lemma \ref{eta_exceeds_delta_left}).  Thus, $\kappa(v)  \st{>} 0$.
 \end{proof}

\begin{proposition}
\label{just_delta}
There is only one point $\delta \in [0, \beta)$ so that $\gamma\,'(\delta)  \eeq   (-1,0)$.
\end{proposition}

\begin{proof}
\indent Suppose for contradiction that  $U - \{v\}$ is nonempty.  Since \( \gamma \, '(v) \eeq (-1,0) \) and  \( \kappa(v) > 0 \), there exists $\varepsilon  \st{>}0$ so that $\gamma\,'$ is strictly in the second quadrant on $(v-\varepsilon,v)$ and $\gamma\,'$ is strictly in the third quadrant on $(v,v  \pl  \varepsilon)$.  Since
$\gamma\,'$ is strictly in the second quadrant on $(v-\varepsilon, v)$,  $U-\{v\}  \eeq   \{s \in [0, v-\varepsilon]: \gamma\,'(s)  \eeq   (-1,0)\}$; that is,
$U-\{v\}$ is a level set of the restriction of $\gamma\,'$ to $[0,v-\varepsilon]$.  As such, $U - \{v\}$ is closed in $[0, v-\varepsilon]$, which means that $U - \{v\}$ is a compact subset of $\R$ and has a maximum element $u$.\\
\indent We claim that $\gamma\,'(s)$ is strictly in the second quadrant for all $s \in (u,v)$.  To prove so, suppose for contradiction that there exists \( s \st{ \in} (u,v) \) so that \( \gamma \, '(s) \) is not strictly in the second quadrant.  By Lemma \ref{obvious}, \( \gamma_{\, 1}(s) \st{>} 0 \).  Hence, we apply Lemma \ref{tangent vector points left} to conclude that \( \gamma \, '(s) \) is in the third quadrant.  Since  \( \gamma_{\,1} \,' \st{<} 0 \) on \((0, \delta] \) (Lemma \ref{tangent vector points left}) and \( \gamma \, ' \) is strictly in the second quadrant on \(( v \mn \varepsilon, v) \), there exists \( t \st{\in} [s, v) \) so that \( \gamma \, '(t) \eeq (-1,0) \), contradicting maximality of \( u \) in \( U - \{v\} \).\\
\indent We define $w$ to be the unique point in $(v,\beta)$ so that $\gamma_{\,2}(w)  \eeq   \gamma_{\,2}(u)$.  
We will ultimately achieve a contradiction by showing that $\gamma\,'(w)  \eeq   (-1,0)$.  In turn, we will accomplish this by curvature comparison.  Let $s \in (v,w)$, and let $\ol{s}$ be the unique point in $(u,v)$ so that $\gamma_{\,2}(\ol{s})  \eeq   \gamma_{\,2}(s)$.  We claim that $\kappa(s)  \leeq  \kappa(\ol{s})$.  Since $\kappa'(v)  \st{<}  0$ (Lemma \ref{fact_3}), we already know that this inequality holds for all $s$ sufficiently close to $v$.
Additionally, recall that there exists $\varepsilon  \st{>}0$ so that $\gamma\,'(s)$ is strictly in the third quadrant for all $s$ in $(v, v \pl  \varepsilon)$.  We will prove that $\gamma\,'(s)$ is strictly in the third quadrant for all $s \in (v,w)$.\\
\indent Let $W  \eeq   \{s \in (v,w): \gamma\,'(t)$ is strictly in the third quadrant and $\kappa(t)  \leeq  \kappa(\ol{t})$ for all $t$ in $(v,s]\}$. Since  \(W \) is nonempty and bounded above,  \( W \) has a supremum, which we shall denote by \(z\).   As in Proposition \ref{3.13 analogue, left case}, the following inequalities hold for all $s$ in $(v,z)$:
\begin{equation}
\label{gamma_one_ineq_2}
\gamma_{\, 1}(\ol{s}) - \gamma_{\, 1}(v)  \leeq  \gamma_{\, 1}(v) - \gamma_{\, 1}(s),
\end{equation}

\begin{equation}
\label{theta_ineq_2}
\theta(\gamma\,'(s))  \leeq  2 \pi-\theta(\gamma\,'(\ol{s})).
\end{equation}
It can also be proved that $\lambda(s)  \geeq  \lambda(\ol{s})$ for all $s$ in $(v,z)$.
By continuity of all relevant quantities on $(0, \beta)$, it follows that these inequalities hold at $z$ as well.\\
\indent Finally, since \( \gamma \,' \) is strictly in the second quadrant on \((u,v) \) and strictly in the third quadrant on \(W\), it can be proved by a similar argument to that in Proposition \ref{kappa comparison, left}  that $w  \eeq   z$.  It follows that the inequalities (\ref{gamma_one_ineq_2}) and (\ref{theta_ineq_2}) hold for all $s$ in $(v,w]$.  By (\ref{theta_ineq_2}),  $\theta(\gamma\,'(w))  \leeq  2 \pi \mn \theta(\gamma\,'(\ol{w})) \eeq   2 \pi \mn \theta(\gamma\,'(u))  \eeq   \pi$.  Since $\theta \circ \gamma\,' \in (\pi, 3\pi/2)$ on $(v,w)$, it must be the case that $\theta(\gamma\,'(w))  \eeq   \pi$.  That is,
 $\gamma\,'(w)  \eeq   (-1,0)$, contradicting the fact that there exists no $s  \st{>} v$ with $\gamma\,'(s)  \eeq   (-1,0)$.
\end{proof}

\section{Glossary of Notation}
Throughout this section, we assume, as at the beginning of Section 3, that $E$ is a spherically symmetric 
 isoperimetric region, and that $A \subset \R^2$ is a spherically symmetric set that generates $E$ when rotated about the $e_1$-axis.  We first summarize the meanings that we have assigned to characters of the Latin alphabet, then proceed through the characters of the Greek alphabet that are used in the article.  Characters used only in Section 7.1 are excluded.
 
 \begin{longtable}[l]{p{50pt}  p{352pt}} 
$A_s$ & Given an $s \in (-\beta, \beta)$, $A_s$ denotes the osculating circle to $\gamma$ at $\gamma(s)$ (see Defn. \ref{as parameterized}). \\
\, & \, \\
$C_s$ & Given an $s \in (-\beta, \beta)$, $C_s$ denotes the canonical circle to $\gamma$ at $\gamma(s)$, i.e. the unique oriented circle that is 
tangent to $\gamma$ at $\gamma(s)$ and has its center on the $e_1$-axis (see Defn. \ref{canonical_circle}). \\
\, & \, \\
$\ti{C}_t$ & For a fixed $s$, let $\alpha$ be an arclength parameterization of $A_s$.  Given $t$ in the domain of $\alpha$ such that $\alpha_2(t) \neeq 0$ or $\alpha'(t) \eeq (0, \pm 1)$, $\ti{C}_t$ denotes the canonical circle to  $\alpha$ at $\alpha(t)$, i.e. the unique oriented circle that is 
tangent to $\alpha$ at $\alpha(t)$ and has its center on the $e_1$-axis (see Defn. \ref{tilde quantities}). \\
\, & \, \\
$F$ & Given $s \in (-\beta, \beta)$, $F(s)$ denotes the abscissa of the center of $C_s$ (see Defn. \ref{canonical_circle}).\\
\, & \, \\
$\ti{F}$ & For a fixed $s$, let $\alpha$ be an arclength parameterization of $A_s$. For any $t$ such that $\ti{C}_t$ exists, $\ti{F}(t)$ denotes the abscissa of the center of $\ti{C}_t$ (see Defn. \ref{tilde quantities}).\\
\, & \, \\
$G$ & We define the function $G$ on $(-\beta, \beta)$ by $G(s) \eeq F(s) \mn R(s)$ (see Defn. \ref{DefininG}). \\
\, & \, \\
$\ti{G}$ & For a fixed $s$, let $\alpha$ be an arclength parameterization of $A_s$. For any $t$ such that $\ti{C}_t$ exists, let $\ti{G}(t) \eeq \ti{F}(t) \mn \ti{R}(t)$ (see Defn. \ref{DefininG}).\\
\, & \, \\
$H_0$ & Given a regular point $x \in \partial E$, $H_0(x)$ denotes the unaveraged mean curvature of $\partial E$ at $x$ 
(i.e. the sum of the principal curvatures of $\partial E$ at $x$).  After parameterizing (the rightmost component of) $\partial A$, we also 
 consider $H_0$ as a function of arclength: given $s \in (-\beta, \beta)$, we let $H_0(s)$ denote the unaveraged  
 mean curvature of $\partial E$ at $\gamma(s)$ (see Defn. \ref{GMC}).\\
 \, & \, \\
$H_1$ &  Given a regular point $x \in \partial E$, $H_1(x)$ denotes the directional derivative of the log of the density 
function in the direction of the outward unit normal vector to $\partial E$ at $x$.  Meanwhile, given 
$s \in (-\beta, \beta)$, we let $H_1(s)$ denote the directional derivative of the log of the density function in 
the direction of the outward unit normal vector to $\partial E$ at $\gamma(s)$ (see Defn. \ref{GMC}).  \\
\, & \, \\
$\ti{H_1}$ &  For a fixed $s$, let $\alpha$ be an arclength parameterization of $A_s$.  For each $t$ in the domain of $\alpha$, let
\[
\ti{H_1}(t) \eeq \frac{p}{|\alpha(t)|} \frac{\alpha(t)}{|\alpha(t)|} \cdot \nu(t),
\]
where $\nu(t)$ is the outward unit normal vector to $\alpha$ at $\alpha(t)$ (see Defn. \ref{as parameterized}). \\
\, & \, \\
$H_f$ & Given a regular point $x \in \partial E$, $H_f$ denotes the generalized mean curvature of $\partial E$ at $x$.  Given 
$s \in (-\beta, \beta)$, we let $H_f(s)$ denote the generalized mean curvature of $\partial E$ at $\gamma(s)$ (see Defn. \ref{GMC}). \\
\, & \, \\
$h$ & In both cases, $h$ denotes a local inverse function for $\gamma_{\,2}$ with codomain $(0, \delta)$:  if $y \in (\gamma_{\,2}(\eta), \gamma_{\,2}(\delta))$, then $h(y)$ is the unique $t \in (0, \delta)$ so that $\gamma_{\,2}(t) \eeq y$ (see Defn. \ref{local_inverses}). \\
\, & \, \\
$K$ &In both cases, $K$ denotes the subset of $[0, \beta)$ that we call the upper curve.  In the right case, 
the upper curve is defined as the set of $s \in [0, \beta)$ so that  $\gamma\,'(t)$ lies in the second quadrant and 
$\kappa(t)  \geeq \lambda(t)  \st{>} 0$ for all $t \in [0,s]$  (see Defn. \ref{upper_curve_right}).  In the left case, the upper curve is defined as 
 the set of $s \in [0, \beta)$ so that $\gamma\,'(t)$ is strictly in the second quadrant for all $t \in (0,s)$ (see Defn. \ref{upper_curve_left}). \\
 \, & \, \\
 $k$ & In both cases, $k$ denotes a local inverse function for $\gamma_{\,2}$ with codomain $(\delta, \eta)$:  if $y \in (\gamma_{\,2}(\eta), \gamma_{\,2}(\delta))$, then $k(y)$ is the unique $t \in (\delta, \eta)$ so that $\gamma_{\,2}(t) \eeq y$ (see Defn. \ref{local_inverses}). \\
\, & \, \\
$L$ & In both cases,  $L$ denotes the subinterval of $[0, \beta)$ that we call the lower curve.  In each case, the definition of $L$ is rather technical, so we refer the reader to Definition \ref{lower_curve_right} in the right case (Section 6) and to Definition \ref{lower_left} in the left case (Section 7).\\
\, & \, \\
$R$ & Given $s \in (-\beta, \beta)$, we let $R(s)$ denote the radius of $C_s$ 
(see Defn. \ref{canonical_circle}). \\
\, & \, \\
$\ti{R}$ & For a fixed $s$, let $\alpha$ be an arclength parameterization of $A_s$. For any $t$ such that $\ti{C}_t$ exists, we let $\ti{R}(t)$ denote the radius of $\ti{C}_t$ (see Defn. \ref{tilde quantities}).\\
\, & \, \\
$\ol{s}$ & In each case, if $s \in L$, we let $\ol{s}$ denote the unique point in $K$ so that $\gamma_{\, 2}(\ol{s}) \eeq \gamma_{\,2}(s)$ (see Prop. 6.13, Prop. 7.16). \\
\, & \, \\
$\ti{s}$ & For a fixed $s$, let $\alpha$ be an arclength parameterization of $A_s$.  We let $\ti{s}$ denote the point in the domain of $\alpha$ such that $\alpha(\ti{s}) \eeq \gamma(s)$ (see Defn. \ref{as parameterized}). \\
\, & \, \\
$\alpha$ & For a fixed $s$, we let $\alpha$ denote an arclength parameterization of $A_s$ (see Defn. \ref{as parameterized}). \\
\, & \, \\
$\pm \beta$ & Endpoints of the domain of $\gamma$ \\
 \, & \, \\
$\gamma$
& Denotes an arclength parameterization of a component of $\partial A$ (which, in fact, turns out to be the only component of $\partial A$; see the beginning of Section 3). \\
 \, & \, \\
$\delta$ & In each case, $\delta$ denotes the supremum of the upper curve. (In the right case (Section 6), see  
Defn. \ref{upper_curve_right} and following. In the left case (Section 7), see Defn. \ref{upper_curve_left} and following.) \\
\, & \, \\ 
$\eta$ & In each case, $\eta$ denotes the supremum of the lower curve. (In the right case (Section 6), see Defn. \ref{lower_curve_right}.  In the left case (Section 7), see Defn. \ref{lower_left}.) \\
\, & \, \\
$\theta$ & We define $\theta: S^{1} \to (0, 2\pi]$ by letting $\theta(v)$ be the angle in the specified interval that $v$ makes with the positive $e_1$-axis (see Defn. \ref{theta definition}). \\
\, & \, \\
$\kappa$ & Given $s \in (-\beta, \beta)$, $\kappa(s)$ denotes the signed curvature of $\gamma$ at $\gamma(s)$. \\
\, & \, \\
$\ti{\kappa}$ & For a fixed $s$, let $\alpha$ be an arclength parameterization of $A_s$. For any $t$ in the domain of $\alpha$, we let $\ti{\kappa}(t)$ denote the signed curvature of $A_s$ at $\alpha(t)$ (see Defn. \ref{as parameterized}). \\
\, & \, \\
$\lambda$ & Given $s \in (-\beta, \beta)$, $\lambda(s)$ denotes the signed curvature of $C_s$ (see Defn. \ref{canonical_circle}).  \\
\, & \, \\
$\ti{\lambda}$ & For a fixed $s$, let $\alpha$ be an arclength parameterization of $A_s$. For any $t$ such that $\ti{C}_t$ exists, $\ti{\lambda}(t)$ denotes the signed curvature of $\ti{C}_t$ (see Defn. \ref{tilde quantities}).  
\end{longtable}

\section*{Acknowledgements}
\noindent This paper is the work of Gregory Chambers with the 2014 Williams College NSF ``SMALL'' Geometry Group, advised by Frank Morgan, and was completed in undergraduate thesis work by Tammen with Ted Shifrin at the University of Georgia. We would like to thank  the NSF, Williams College, and the MAA for supporting the ``SMALL'' REU, Chambers' visit to Williams, and our travel to MathFest.   We would also like to thank the anonymous referee who gave us many helpful suggestions. \\

\bibliographystyle{abbrv}

\bigskip
\noindent Wyatt Boyer
\newline 
Department of Mathematics 
\newline
Boston College
\newline
boyerw@bc.edu
\newline
\newline
Bryan Brown
\newline
Department of Mathematics
\newline
Pomona College
\newline
bcb02011@mymail.pomona.edu
\newline
\newline
Gregory R. Chambers
\newline
Department of Mathematics
\newline
University of Chicago
\newline
chambers@math.uchicago.edu
\newline
\newline
Alyssa Loving
\newline
Department of Mathematics
\newline
University of Illinois Urbana-Champaign
\newline
aloving2@illinois.edu
\newline
\newline
Sarah Tammen
\newline
Department of Mathematics
\newline
Massachusetts Institute of Technology
\newline
setammen@mit.edu

\end{document}